\documentclass[11pt]{amsart}
\usepackage{geometry}                
\usepackage{graphicx}
\usepackage{amssymb}
\usepackage{amscd}
\usepackage{epstopdf}
\title{Filtered $F$-Crystals on Shimura Varieties of Abelian Type}
\author{Tom Lovering}
\newtheorem{thm}[subsubsection]{Theorem}
\newtheorem*{thm*}{Theorem}
\newtheorem{lem}[subsubsection]{Lemma}
\newtheorem{prop}[subsubsection]{Proposition}
\newtheorem{cor}[subsubsection]{Corollary}

\newcommand{\A}{\mathbb{A}}
\newcommand{\Af}{\mathbb{A}^\infty}
\newcommand{\Afp}{\mathbb{A}^{\infty,p}}
\newcommand{\C}{\mathbb{C}}

\newcommand{\D}{\mathbb{D}}
\newcommand{\Fp}{\mathbb{F}_p}
\newcommand{\G}{\mathbb{G}}
\newcommand{\bG}{\mathbb{G}}

\newcommand{\Gm}{\mathbb{G}_m}
\newcommand{\bL}{\mathbb{L}}
\newcommand{\bM}{\mathbb{M}}

\newcommand{\bP}{\mathbb{P}}
\newcommand{\Q}{\mathbb{Q}}

\newcommand{\Qp}{\mathbb{Q}_p}

\newcommand{\R}{\mathbb{R}}
\newcommand{\bS}{\mathbb{S}}

\newcommand{\Z}{\mathbb{Z}}
\newcommand{\Zp}{{\mathbb{Z}_p}}
\newcommand{\Zbp}{{\mathbb{Z}_{(p)}}}

\newcommand{\cA}{\mathcal{A}}
\newcommand{\cB}{\mathcal{B}}
\newcommand{\cC}{\mathcal{C}}

\newcommand{\cE}{\mathcal{E}}
\newcommand{\cF}{\mathcal{F}}
\newcommand{\cG}{\mathcal{G}}
\newcommand{\cH}{\mathcal{H}}

\newcommand{\cL}{\mathcal{L}}
\newcommand{\cM}{\mathcal{M}}

\newcommand{\cO}{\mathcal{O}}
\newcommand{\cP}{\mathcal{P}}

\newcommand{\cR}{\mathcal{R}}
\newcommand{\cS}{\mathcal{S}}
\newcommand{\cT}{\mathcal{T}}

\newcommand{\cV}{\mathcal{V}}

\newcommand{\fM}{\mathfrak{M}}

\newcommand{\fS}{\mathfrak{S}}
\newcommand{\fX}{\mathfrak{X}}
\newcommand{\fY}{\mathfrak{Y}}

\newcommand{\Aut}{\operatorname{Aut}}
\newcommand{\Spec}{\operatorname{Spec}}
\newcommand{\Spf}{\operatorname{Spf}}
\newcommand{\Spa}{\operatorname{Spa}}
\newcommand{\Hom}{\operatorname{Hom}}
\newcommand{\Ker}{\operatorname{Ker}}

\newcommand{\Fil}{\operatorname{Fil}}
\newcommand{\gr}{\operatorname{gr}}

\newcommand{\id}{\operatorname{id}}

\newcommand{\Isom}{\operatorname{Isom}}
\newcommand{\uHom}{\underline{\Hom}}
\newcommand{\uIsom}{\underline{\Isom}}
\newcommand{\Rep}{\operatorname{Rep}}
\newcommand{\vVec}{\operatorname{Vec}}
\newcommand{\Bun}{\operatorname{Vec}}
\newcommand{\Lisse}{\operatorname{Lisse}}
\newcommand{\Mod}{\operatorname{Mod}}
\newcommand{\Par}{\operatorname{Par}}

\newcommand{\Fib}{\operatorname{Fib}}
\newcommand{\Rees}{\operatorname{Rees}}

\newcommand{\mult}{\operatorname{mult}}

\newcommand{\FCrys}{\operatorname{FFCrys}}
\newcommand{\Res}{\operatorname{Res}}
\newcommand{\Gal}{\operatorname{Gal}}
\newcommand{\Sh}{\operatorname{Sh}}

\newcommand{\nocontentsline}[3]{}
\newcommand{\tocless}[2]{\bgroup\let\addcontentsline=\nocontentsline#1{#2}\egroup}

\newcommand{\rightiso}{\stackrel{\cong}{\rightarrow}}
\newcommand{\leftiso}{\stackrel{\cong}{\leftarrow}}
\newcommand{\map}[1]{\stackrel{#1}{\rightarrow}}
\newcommand{\lmap}[1]{\stackrel{#1}{\leftarrow}}

\usepackage{hyperref}
\hypersetup{
    linktoc=all,
    bookmarksnumbered
}

\begin{document}
\maketitle

\begin{abstract}
In this paper, we define and construct canonical filtered $F$-crystals with $G$-structure over the integral models for Shimura varieties of abelian type at hyperspecial level defined by Kisin \cite{kis2}. We check that these are related by $p$-adic comparison theorems to the usual lisse sheaves, and as an application we also use this to show that the Galois representations generated from the $p$-adic \'etale cohomology of Shimura varieties with nontrivial coefficient sheaves are crystalline, at least in the case of proper abelian type Shimura varieties.
\end{abstract}

\tableofcontents
\section{Introduction}
Shimura varieties are rich algebro-geometric objects straddling the gulf between the worlds of number theory on the one hand and representation theory on the other. As such, they play a central role in our current understanding of the Langlands programme, and this paper attempts to address some gaps in our current understanding of their geometry and p-adic Hodge theory.

It has long been understood that a useful way to think about Shimura varieties is as moduli spaces of ``motives with $G$-structure.'' In the classical PEL type case this heuristic has enjoyed success as a concrete moduli problem, and later Milne \cite{milne4} managed to write down a moduli problem for a much wider range of Shimura varieties. Unfortunately, outside this range it seems something more general than the usual notion of a rational Hodge structure (and hence a motive) is required, and Milne's description also fails to immediately yield a tidy story when extending to integral models.

Nevertheless, there are still things one can say. For example (putting aside mild technical constraints \ref{3.1.3}), given a point $x \in \Sh_K(G,\fX)(F)$ of a Shimura variety, one may use the tower $\Sh(G,\fX) \rightarrow \Sh_K(G,\fX)$, viewed as a pro-\'etale cover, to manufacture for all $p$ a fibre functor taking values in $p$-adic Galois representations
$$\omega_{et,x}: \Rep_{\Qp}(G_{\Qp}) \rightarrow \Rep_{\Qp}(\Gamma_F),$$
an object it is not misleading to think of as the $p$-adic \'etale cohomology of a `motive' attached to the point $x$. It is also possible to give archimedean (e.g. \cite{milne3}) and nonarchimedean \cite{lz} analytic constructions of fibre functors one might similarly view as the de Rham cohomology of the `motive' attached to $x$, at least in some settings. The problem of constructing integral models has also been addressed by Kisin, at least the smooth models in the hyperspecial abelian type case \cite{kis2}.

In this paper, building on the nonarchimedean de Rham construction of Liu-Zhu \cite{lz}, we construct a similar functor
$$\omega_{crys,x}: \Rep_{\Zp}(G_{\Zp}) \rightarrow \FCrys_{x/W(\kappa(x))}$$
which one should view as a stand-in for the \emph{integral crystalline cohomology} of the `motive' attached to an integral point $x$ of Kisin's integral models. We then give an obvious consequence for the Galois representations that turn up in the conjectural construction of the Langlands correspondence: namely that they are often crystalline when one expects them to be, and we also hope our construction will prove useful in studying the geometry of these integral models, such as in recent work of Hamacher \cite{ham}.

Let us now go into more detail and explain the main results of our paper. We begin in \S{\ref{S2}} with technical preliminaries. We carefully develop the theory of filtered $F$-crystals with $G$-structure on a smooth formal scheme $\fX/W(\kappa)$ where $W(\kappa)$ is the ring of Witt vectors over a finite field $\kappa$. While one may define such objects relatively painlessly as fibre functors taking values in filtered $F$-crystals
$$\omega: \Rep_{\Zp}(G) \rightarrow \FCrys_{\fX/W}$$
there are some technical subtleties worth reviewing, and it will be useful to have in hand a more geometric description of $\omega$ as corresponding to a $G$-bundle equipped with various extra structures. We also review the basics of the theory of $\fS$-modules and some $p$-adic comparison theorems together with results that relate them to the $G$-structure setting.

In \S{\ref{S3}} we get to the heart of the paper and give our results on Shimura varieties. Let $(G,\fX)$ be a Shimura datum with reflex field $E=E(G,\fX)$ and fix a place $v$ of $E$ dividing some rational prime $p$. Assume that $G$ is unramified at $p$ and fix $G/\Zbp$ a reductive model, let $K_p=G(\Zp)$, and assume that $(G,\fX)$ is of abelian type. Then \cite{kis2} gives a system of smooth integral models $\cS_{K_pK^p}/\cO_v$ for the Shimura varieties $\Sh_{K_pK^p}(G,\fX)$ as $K^p$ varies over sufficiently small open compact subgroups of $G(\Afp)$. 

Let $G^c$ be the quotient of $G$ described in (\ref{3.1.3}) obtained by killing the part of the centre that splits over $\R$ but not over $\Q$. Then it is well-known that the tower $\Sh_{K^p} \rightarrow \Sh_{K_pK^p}$ is a pro-(finite \'etale) $G^c(\Zp)$-torsor and so gives rise to a fibre functor
$$\omega_{et,K^p}: \Rep_{\Zp}(G^c) \rightarrow \Lisse_{\Zp}(\Sh_{K_pK^p}).$$
Combining the main results of Liu-Zhu \cite{lz} with the observation that the restriction of $\omega_{et,K^p}$ to special points may be checked explicitly to be de Rham, one may produce a de Rham fibre functor on the rigid generic fibre $\Sh_{K_pK^p}^{an}$ of $\cS_{K_pK^p}$ taking values in filtered vector bundles with connection
$$\omega_{dR,K^p}: \Rep_{\Qp}(G^c) \rightarrow \Fil^\nabla(\Sh_{K_pK^p}^{an}/E_v)$$
compatible with $\omega_{et,K^p}$ via $p$-adic Hodge theory.

We define (\ref{3.1.5}) a \emph{crystalline canonical model} of this functor to be a fibre functor taking values in strongly divisible filtered $F$-crystals
$$\omega_{crys, K^p}: \Rep_{\Zp}(G^c) \rightarrow \FCrys_{\cS_{K_pK^p}/\cO_v}$$
together with an identification $\iota: \omega_{crys,K^p}[1/p] \rightiso \omega_{dR,K^p}$ of fibre functors taking values in filtered vector bundles with connection satisfying what we call the CPLF conditions. Roughly speaking, this says that at unramified points $x$ of the integral model with $x^*\omega_{et}$ taking values in crystalline representations, there is a lattice condition asserting that the lattices given by  $x^*(\omega_{crys},\iota)$ must agree with some lattices constructed from $x^*\omega_{et}$ via the theory of $\fS$-modules, and a similar Frobenius condition relating the Frobenius attached to $x^*\omega_{crys}[1/p]$ with that coming from Fontaine's $D_{crys}$ functor. If we have all the features of such a model except the guarantee that the lattices produced are strongly divisible we call it a \emph{weak crystalline canonical model}. It is easy to check such objects if they exist are unique, and our main theorem is the following.

\begin{thm*}
With the setup above, in particular with $(G,\fX)$ of abelian type, if $p>2$ then a crystalline canonical model exists. If $p=2$\footnote{We should remark that the caveat at $p=2$ may likely be removed. The missing ingredient is a proof of (\ref{special strong}) in this case. Since strong divisibility seems like a less interesting condition at $p=2$ for our purposes, we did not search very hard for such a proof, which may very well exist and be a fun exercise for somebody interested in such things.}, a weak crystalline canonical model exists.
\end{thm*}

The proof of this theorem is very similar to that of the main theorem of our previous paper \cite{tl}, except we now are working locally at a single prime so can use Kisin's models directly, and rather than using properties of Milne's de Rham construction we must use properties of that of Liu and Zhu. Apart from making the notation more generally tidy, this has two advantages. First, the assumption of \cite{tl} that $Z(G)^\circ$ be split by a CM field may now be removed, since we may use pure $p$-adic Hodge theory in place of the theory of CM motives, which underpin Milne's work and the constructions of \cite{tl}. Second, combining the present paper with \cite{tl} and noting that the constructions are the same gives a (perhaps somewhat roundabout) proof of Liu-Zhu's conjecture \cite[4.9 (ii)]{lz} that the $p$-adic analytification of Milne's construction agrees with theirs in the case of abelian type Shimura varieties (with $Z(G)^\circ$ split by a CM field of course).

Let us briefly recall the shape of it. First, one can use the theory of classical crystalline representations and $\fS$-modules directly to handle the ``special type'' case where $G$ is a torus. Next, one can use a Hodge embedding $G \hookrightarrow GSp(V)$ and carefully study the universal abelian scheme and a family of tensors on its sheaf of de Rham cohomology $s_{\alpha,dR} \in \hat{\cV}^\otimes$ to produce a geometric construction of a $G=G^c$-bundle
$$\cP_{K_pK^p} := \uIsom_{s_{\alpha}}(V,\hat{\cV})$$
which one checks has all the bells and whistles necessary to define a filtered $F$-crystal with $G$-structure. This relies on an understanding of Kisin's argument \cite[2.3]{kis2} and some of our technical preliminaries. Finally, one can pass from the Hodge $(G,\fX)$ to abelian $(G_2,\fX_2)$ type case by means of an auxiliary Shimura datum $(\cB,\fX_{\cB})$ attached to the group $\cB = G \times_{G^{ab}} E^*$ that sits in a diagram
$$(G,\fX) \leftarrow (\cB,\fX_{\cB}) \rightarrow (G_2,\fX_2).$$
The first stage is to pass from $G$ to $\cB$ by an explicit construction using our knowledge of the special and Hodge types. The second is to pass from $\cB$ to $G_2$ and happens via a descent argument part of which involves a pushout of bundles along $\cB^c \rightarrow G_2^c$.

Since the Hecke $G(\Afp)$ acts transitively on the set of components at infinite level, unlike $\prod_{l|N}G(\Q_l)$ in \cite{tl}, we remark that the descent operation to get from $\Sh_{\cB(\Zp)}(\cB,\fX_{\cB})$ to $\Sh_{G_2(\Zp)}(G_2,\fX_2)$ can be made to look rather more group theoretic and like Deligne's original formalism \cite{del2} for the construction of canonical models than the construction of our previous paper.

One obvious expectation, given our motivation of this construction in terms of motives, is that the lisse $\Zp$ sheaves $\omega_{et}$ and the filtered $F$-crystals $\omega_{crys}$ ought to be related by a $p$-adic comparison isomorphism. We therefore check this  (\ref{assoc thm}) in the case of the well-known theorems of Faltings \cite{fal1}. One corollary of this fact, in fact the original motivation for this project, is the following. Recall that the for $\rho \in \Rep(G)$ one roughly expects certain automorphic forms of weight $\rho$ and level $K$ to have associated Galois representations lying in cohomology groups of the form $H^i(\Sh_K(G),\omega_{et}(\rho))$. It is well-known and easy to check that at hyperspecial level $l \not=p$ these cohomology groups are unramified. Our argument gives the analogous result at $l=p$, at least in the proper case.

\begin{thm*}
Let $(G,\fX)$ be a Shimura datum of abelian type. Suppose an open compact $K \subset G(\Af)$ is hyperspecial at $p$, and $\rho: G^c_{\Zbp} \rightarrow GL(V_{\Zbp})$ be an algebraic representation. 

Assume $\Sh_K(G,\fX)$ is proper. Then the $p$-adic Galois representation
$$H^i(\Sh_K(G,\fX)_{\bar{E}}, \omega_{et}(\rho_\Zp)[1/p])$$ is crystalline at all places $v|p$ of $E$.
\end{thm*}

We also expect our argument to extend with a little more work to similar non-proper situations, and include an outline of how we expect it to go, in particular our hope that a serious such project could be built on current work of Madapusi Pera \cite{mp}.

\section*{Acknowledgements}
The author cannot underscore enough the debt this work owes to his PhD supervisor Mark Kisin. His constant clarity of thought and dedicated patience were vital to its eventual success. Work on this paper also began after a research idea Mark Kisin proposed interacted with another idea that arose in a conversation with Jack Thorne, and suggested this would be an interesting endeavour.

For other useful conversations we thank Chris Blake, George Boxer, Lukas Brantner, Justin Campbell, Erick Knight, Sam Raskin, Ananth Shankar, Tony Scholl, Koji Shimizu, Jack Shotton, Rong Zhou, and Yihang Zhu. Also, I should thank all participants in the Harvard 2013 \emph{Automorphic Forms and Galois Representations} learning seminar where my interest for this particular subject began.

This paper is the second part of the author's PhD thesis completed at and funded by Harvard University, and was begun while the author was on a Kennedy Scholarship.

\section{Filtered $F$-crystals with $G$-structure}
\label{S2}
\subsection{Rees construction}
We proceed to present a perspective on the notion of ``filtered bundles'' that is amenable to a Tannakian formalism and allows one to speak of and manipulate ``filtered $G$-bundles'': namely the Rees construction. In the context of crystalline cohomology it also interacts with Frobenius in a neat way we will see shortly. We suspect much of this is known to the experts, but were unable to find a good reference so develop some of the necessary results from scratch.

\subsubsection{}
Let $R$ be a Noetherian ring. A \emph{flatly filtered bundle over $R$} is a flat finitely generated $R$-module $N$ together with an exhaustive decreasing filtration $\Fil^i$ of $N$ with the property that its associated graded
$$\text{gr}^\bullet_{\Fil} N := \oplus_i \Fil^{i}/\Fil^{i+1}$$
is again flat as an $R$-module.

\subsubsection{}
There are various ways to think about equivariant objects: for now we record the two points of view that are most useful for us. First recall the standard abstract definition. Let $X$ be a scheme, $H$ a group scheme, and $\alpha: H \times X \rightarrow X$ an action of $H$ on $X$. Then to give an $H$-equivariant $X$-scheme\footnote{One may replace the word ``$X$-scheme" here with any notion that sits in a category fibred over $X$-schemes, and all the translations are obvious.} $Y$ is to supply:
\begin{itemize}
\item An $X$-scheme $Y \rightarrow X$.
\item An isomorphism of $H \times X$-schemes\footnote{We may write this as an isomorphism $pr_2^*Y \rightiso \alpha^*Y$ if it is more convenient. This notation often makes our exploitation of various functorialities more apparent.}, 
$$\beta: (H \times X ) \times_{pr_2,X} Y \rightiso (H \times X) \times_{\alpha, X} Y$$
satisfying the following explicit cocycle condition.
\end{itemize}
There are three natural maps $H \times H \times X \rightarrow X$:
$$pr_3 = pr_2 \circ (m_H \times id_X) = pr_2 \circ pr_{23}: H \times H \times X \rightarrow X.$$
$$A:= \alpha \circ (m_H \times id_X) = \alpha \circ (id_H \times \alpha): H \times H \times X \rightarrow X.$$
$$B := pr_2 \circ (id_H \times \alpha) = \alpha \circ pr_{23}: H \times H \times X \rightarrow X.$$

We may construct from $\beta$ a triangle of maps
$$pr_3^*Y \cong pr_{23}^* pr_2^*Y \map{pr_{23}^*\beta} pr_{23}^*\alpha^* Y \cong B^*Y,$$
$$B^*Y \cong (id_H \times \alpha)^* pr_2^*Y \map{(id_H \times \alpha)^*\beta} (id_H \times \alpha)^* \alpha^*Y \cong A^*Y$$
 and
$$pr_3^*Y \cong (m_H \times id_X)^* pr_2^*Y \map{(m_H \times id_X)^* \beta} (m_H \times id_X)^* \alpha^*Y \cong A^*Y.$$

The cocycle condition asserts that this triangle commutes. 

This abstract definition is useful for checking quickly that equivariant objects can be translated between different suitably functorial equivalences of categories, but the following will be more useful for getting our hands on $\Gm$-equivariant objects over $\A^1$. 

Let $U \subset \A^1$ be the complement of $0$ (which as a scheme is isomorphic to $\Gm$). Recall that $U$ is a (trivial) $\Gm$-torsor, so by descent theory there is a natural equivalence between schemes over $S$ and $\Gm$-equivariant schemes over $U$.

\begin{lem}
To give a $\Gm$-equivariant scheme $Y$ over $\A^1_S$, it is equivalent to give the data of
\begin{itemize}
\item An $S$-scheme $Y_1$,
\item An $\A^1_S$-scheme $Y$ such that there exists an identification of $U$-schemes
$$\theta: Y_U \rightiso Y_1 \times_S U$$
with the property that locally there exist generators $g_\alpha$ for $\cO_Y$ such that $\theta(g_\alpha) = h_\alpha \otimes t^{n_{\alpha}}$ for $h_\alpha \in \cO_{Y_1}, n_\alpha \in \Z.$
\end{itemize}
\end{lem}
\begin{proof}
Firstly, given a $\Gm$-equivariant scheme $Y \rightarrow \A^1_S$, we have the datum

$$\beta: (\Gm \times \A^1 ) \times_{pr_2,\A^1} Y \rightiso (\Gm \times \A^1) \times_{\alpha, \A^1} Y.$$

Pulling back along $\Gm \times \{1\} \subset \Gm \times \A^1$, we obtain an isomorphism
$$\theta: Y_{t=1} \times_S U \rightiso Y_U$$
giving us the datum  $(Y_{t=1},Y,\theta)$ required. It has the property of being locally generated by pure monomial tensors because the $\Gm$-action equips $\cO_Y$ with a grading, from which these generators can be read off.

To finish we need to show that given $(Y_1,Y,\theta)$ there is a unique $\Gm$-equivariant structure on $Y$ inducing this datum via the above procedure. Recall that there is an equivalence $Y_1 \mapsto Y_1 \times_S U$ from $S$-schemes to $\Gm$-equivariant $U$-schemes. Since $U$ is dense in $\A^1_S$, uniqueness follows immediately and for existence we are required to show that the canonical $\Gm$-equivariant structure on $Y_U \rightiso Y_1 \times_S U$ can always be extended to $Y$.

This may be checked locally, where it comes down to the following question. Given an $A[t]$-algebra $B$, an $A$-algebra $B_1$ and an isomorphism $\theta: B \otimes_{A[t]} A[t,t^{-1}] \rightiso B_1 \otimes_A A[t,t^{-1}]$, does the natural automorphism 
$$\beta: (b_1 \mapsto b_1, t \mapsto st, s \mapsto s): B_1 \otimes_A A[t,t^{-1},s,s^{-1}] \rightiso B_1 \otimes_A A[t,t^{-1},s,s^{-1}]$$
induce an automorphism of the lattice $\theta(B) \otimes_{A[t]} A[t,s,s^{-1}]$? Since we have the obvious candidates for left and right inverses, it will suffice to check $\beta$ takes $\theta(B)$ into $\theta(B) \otimes_{A[t]} A[t,s,s^{-1}]$. By assumption (after shrinking $\Spec B$ if necessary) $B$ has generators $b_\alpha$ such that $\theta(b_\alpha) = c_\alpha \otimes t^{n_\alpha}$. These are sent under $\beta$ to $c_\alpha \otimes t^{n_\alpha} s^{n_\alpha} = s^{n_\alpha}\theta(b_\alpha) \in \theta(B) \otimes_{A[t]} A[t,s,s^{-1}]$, as required. 
\end{proof}

\subsubsection{}
Our key definition is the following. A \emph{Rees bundle} over $R$ is a flat finitely generated $R[t]$-module $M$ together with a $\mathbb{G}_{m,R}$-action equivariant for the standard action of $\Gm$ on $\A^1$. 

We remark that one should visualise this as a vector bundle $\cV$ on $\A^1 \times \Spec R$ with an equivariant $\Gm$ action, and the fibre $\cV_{t=1}$ at $1$ as a `filtered bundle' with the filtration given by using the $\Gm$ action to spread out $v \in \cV_{t=1}$ to an invariant section $\tilde{v} \in \cV|_{\Gm}$ and saying that $v \in \Fil^q$ if $\tilde{v}$ vanishes to order at least $q$ at $t=0$. Hopefully with this in mind, the following should be intuitive and help make the idea more precise.

\begin{prop}
\label{rees equivalence}
There is a natural equivalence of exact rigid $\otimes$-categories between flatly filtered bundles over $R$ and Rees bundles over $R$, compabible with base change along a ring map $R \rightarrow S$.
\end{prop}

\begin{proof}
 Given a flatly filtered bundle $(N,\Fil^\bullet)$ we may form its associated Rees bundle
$$\Rees(N) := \sum_i \Fil^i \otimes t^{-i} \subset N \otimes_R R[t,t^{-1}]$$
which (as in the previous lemma) inherits a $\Gm$-equivariant structure from the natural one on the right hand side. 

We must check that $\Rees(N)$ is flat over $R[t]$. Since $N$ is flat, it is clear that $N \otimes_R R[t,t^{-1}] = \Rees(N)[t^{-1}]$ is flat. Thus it suffices to check flatness along a formal neighbourhood of $\{t=0\}$. But
$$\Rees(N)/(t^n) \cong \bigoplus_i \Fil^i/\Fil^{i+n}$$
with multiplication by $t$ shifting elements one index down 
$$ \times t: \Fil^i/\Fil^{i+n} \rightarrow \Fil^i/\Fil^{i+n-1} \subset \Fil^{i-1}/\Fil^{i-1+n}.$$

We need to show this is flat as an $R[t]/(t^n)$-module. Since $R$ is Noetherian we may Zariski-localise and assume everything in sight is free. Freeness of $\text{gr}^\bullet_{\Fil} N$ allows one to prove inductively that we may pick a basis $N \cong R^d$ such that there are integers $d_i$ with the first $d_i$ co-ordinates $R^{d_i} \subset R^d$ giving the submodule $\Fil^i$. With this choice of co-ordinates one has a completely explicit isomorphism
$$(R[t]/(t^n))^d \rightiso \bigoplus_i \Fil^i/\Fil^{i+n}$$
which takes the $j$-th basis element to the corresponding element of $\Fil^p/\Fil^{p+n}$ where $p$ is the greatest number such that $e_j \in \Fil^p$. In particular we see (since it may be checked Zariski locally) that in general $\Rees(N)/(t^n)$ is flat over $R[t]/t^n$.

The inverse is clear: to a Rees bundle $M$ we associate the bundle $N=M/(t-1) \cong M[t^{-1}]^{\Gm}$ filtered by the order of the pole at $t=0$. Since $M$ is flat as an $R[t]$-module, so are $N=M/(t-1)$ and $\gr^\bullet N = M/t$ as $R$-modules.

These operations are obviously inverse, so it suffices to check full-faithfulness in one direction. Given $f:N \rightarrow N'$ a map of flatly filtered bundles it induces a $\Gm$-equivariant map $\Rees(f):\Rees(N) \rightarrow \Rees(N')$ (automatically over $\A^1-\{0\}$ and extending to zero because $f$ respects the filtration). This association is a bijection
$$\Hom(N,N') \cong \Hom(\Rees(N),\Rees(N')).$$
Indeed, the inverse operation can be described: given $g:\Rees(N) \rightarrow \Rees(N')$, $f = g/(t-1): N \rightarrow N'$ is a map of modules and respects the filtration because $g$ extends over $0$.

It is straighforward to check this correspondence respects tensor product and internal hom. If we make a base change of rings $R \rightarrow S$, then we may canonically identify
$$\Rees(N \otimes_R S) = \sum_i (\Fil^i \otimes_R S) \otimes t^{-i} = (\sum_i \Fil^i \otimes t^{-i})\otimes_{R[t]} S[t] = \Rees(N) \otimes_{R[t]} S[t],$$
which proves the claim about base change.
\end{proof}

It makes sense to globalise these objects in the usual way (as coherent sheaves), so given a locally Noetherian scheme $X$ one can talk of the exact rigid $\otimes$-category $\Rees_X$ of \emph{Rees bundles on }$X$, and the above result gives an equivalence between this category and the category of flatly filtered vector bundles which respects base change.

\subsection{Filtered $G$-bundles}
We first recall some basic facts about $G$-bundles and their associated Tannakian formalism, including such results over a Dedekind domain as developed by Broshi \cite{brosh}. We then use the Rees construction to define a notion of a filtered $G$-bundle, and briefly explore the connection between this definition and the definition using Grassmannians as in \cite[\S3]{tl}. 

Note we assume throughout that the base $S$ over which our group is defined is Dedekind, but of course many of the results and definitions are valid in greater generality.

\subsubsection{}
Let $G$ be a flat affine group scheme of finite presentation over a Dedekind base $S$. Then a $G$\emph{-bundle} on an $S$-scheme $X$ is a fpqc morphism $P \rightarrow X$ with a right action of $G$ such that the natural map $(p,g) \mapsto (p,p.g)$
$$P \times_X (G_X) \rightarrow P \times_X P$$
is an isomorphism. Let $\Bun_{X}^G$ be the category (in fact groupoid) of $G$-bundles on $X$.

\subsubsection{}
Now assume $S=\Spec A$ is affine. Let $\Rep_A(G)$ be the category of algebraic representations of $G$ taking values in finite projective $A$-modules (morphisms respecting the $G$-action). This forms a rigid tensor category. We also write $\vVec_X$ for the category of vector bundles on $X$, also a rigid tensor category. We should caution the reader that neither is abelian in general, but both are Quillen-exact\footnote{Recall that for $R$ a commutative ring, an $R$-linear additive category $\cC$ is Quillen-exact iff it is a full additive subcategory of an $R$-linear abelian category $\cC \subset \cC'$ and closed under extensions. It is possible to abstract this notion, viewing a Quillen-exact category as a pair $(\cC,E)$ consisting of an additive category and a distinguished collection $E$ of short exact sequences in $\cC$ satisfying some axioms.
}
 which is enough. The following is standard.

\begin{lem}
Let $P$ be a $G$-bundle on $X$. Then it defines a natural faithful exact $A$-linear tensor functor
$$\omega_P: \Rep_A(G) \rightarrow \vVec_X$$
given by $\omega_P(V) = P \times^G \underline{V}$
which commutes with base change in $X$.
\end{lem}

\subsubsection{}
We note that this proposition has a sort of converse. There is an obvious ``forgetful'' tensor functor 
$$\omega_X: \Rep_A(G) \rightarrow \vVec_A \rightarrow \vVec_X.$$
Recall that a \emph{fibre functor} on an additive exact rigid $A$-linear tensor category $\cC$ is a faithful exact $A$-linear tensor functor $\cC \rightarrow \vVec_X$. Let $\Fib(\cC,\vVec_X)$ denote the category of such functors. The converse is as follows (\cite[1.2]{brosh}).

\begin{thm}
\label{tann}
Suppose $A$ is a Dedekind domain and $G$ is flat affine of finite type over $A$ with connected fibres, $X/A$ a scheme. Then the functor
$$\Bun_X^G \ni P \mapsto \omega_P \in \Fib(\Rep_A(G),\vVec_X)$$
is an equivalence of categories, the inverse being given by
$$F \mapsto \uIsom^\otimes(\omega_X,F).$$
This equivalence is 2-functorial in $X$ and $G$.
\end{thm}

Since this equivalence is functorial, a similar equivalence will be true for many categories richer than $\vVec_X$, with their ``$G$-valued'' equivalent being expressible as a $G$-bundle with extra structures on the one hand and a tensor functor on the other. We shall make liberal use of this setup throughout the paper, including the following.

\subsubsection{}
As above let $S=\Spec A$ be Dedekind, $G/S$ a flat affine finitely presented group scheme, and $X/S$ locally Noetherian. 

We define a \emph{Rees $G$-bundle } on $X$ to be a $G$-bundle $\cP$ on $\A^1 \times X$ together with an equivariant $\Gm$-structure (that commutes with the $G$-action). Given a $G$-bundle $P$ on $X$, a \emph{Rees structure on }$P$ is the data of a Rees $G$-bundle $\cP$ together with an isomorphism $\cP/(t-1) \rightiso P$. Let $\Rees_X^G$ denote the category of Rees $G$-bundles. Again, we have a ``Tannakian'' description of such objects.

\begin{prop}
\label{2.2.7}
Given a Rees $G$-bundle $\cP$ on $X$ and a representation $\rho:G \rightarrow GL(V)$ where $V$ is a finite projective $A$-module, we obtain a canonical Rees bundle
$$ \cV_{\rho} = \cP \times V / G$$
on $X$.
This association $\omega_{\cP}: \rho \mapsto \cV_{\rho}$ is a fibre functor $\Rep_A(G) \rightarrow \Rees_X$, and the map $\cP \mapsto \omega_{\cP}$ gives rise to an equivalence of categories, functorial in $G$ and $X$,
$$\Rees_X^G \cong \Fib(\Rep_A(G), \Rees_X).$$
\end{prop}
\begin{proof}
The first part is immediate from descent theory, noting that since the equivariant $\Gm$-action on $\cP$ commutes with the $G$-action it also descends. That the association is a fibre functor can be seen immediately after composing with the natural $\Rees_X \rightarrow \vVec_{\cP_{t=1}}$.

For the final statement, note that \ref{tann} already gives us the functorial equivalence
$$\vVec_{X \times \A^1}^G \cong \Fib(\Rep_A(G),\vVec_{X \times \A^1}).$$

To put a $\Gm$-equivariant structure on both sides is the same because we also have the equivalence
$$\vVec^G_{\Gm \times X \times \A^1} \cong \Fib(\Rep_A(G),\vVec_{\Gm \times X \times \A^1})$$
and these are functorial with respect to pullback.
\end{proof}

\subsubsection{}
This definition is related to one we made in a previous paper \cite[\S3.3]{tl} defining a notion of filtered $G$-bundle in terms of a map to a flag variety. We record the setup and the comparison result here for completeness. We assume for this discussion that $G$ is connected reductive.

Suppose $A'/A$ is an \'etale cover, and $\mu: \mathbb{G}_{m,A'} \rightarrow G_{A'}$ a cocharacter defined over $A'$. This defines a parabolic $Q_\mu \subset G_{A'}$ corresponding to the nonnegative (for $\mu$) root groups, and we say the \emph{conjugacy class of $\mu$ is defined over $A$} if there is a component $Z_{\mu} \subset \Par_{G/A}$ defined over $A$ with connected geometric generic fibre and which contains the point $Q_\mu \in \Par_{G/A}(A')$. Recall \cite[3.2.2]{tl} that $Z_{\mu}$ parameterises parabolic subgroups of $G$ which are conjugate to $Q_\mu$ \'etale locally, and comes equipped with the natural $G$-action by conjugation.

Then in \cite[3.3]{tl} we defined a $\mu$-filtration on a $G$-bundle $P/X$ to be a $G$-equivariant map $$\gamma: P \rightarrow Z_{\mu}$$
and showed that to give such a datum is to give a fibre functor
$$\omega_\gamma: \Rep_A(G) \rightarrow \Fil_X$$
which \'etale locally looks like the filtration defined by $\mu$. In particular, the filtered bundles so constructed are flatly filtered, so by (\ref{rees equivalence}) and (\ref{2.2.7}) we see that the data of $\gamma$ induces a Rees structure on $P$. Of course the converse is also true provided we can find a suitable cocharacter that \'etale locally captures the shape of a given filtration coming from a Rees bundle.

\subsubsection{}
This description is also easily related to reduction to a parabolic. Indeed suppose $A=A'$, so $\mu$ and $Q_\mu$ are defined over $A$. Then $Z_\mu = G/Q_\mu$ and for $P$ a $G$-torsor we may form a correspondence between $\mu$-filtrations $\gamma: P \rightarrow Z_\mu$ and reductions $(P_\mu,\iota)$ of $P$ to the parabolic subgroup $Q_\mu$, where $P_\mu$ is a $Q_\mu$ torsor and $\iota:P_\mu \times^{Q_{\mu}} G \rightiso P$.

Indeed, given $\gamma$ we may take the fibre along $\gamma$ over $[1] \in G/Q_\mu$ to get a reduction to $Q_\mu$ and conversely given a reduction $P_\mu \times^{Q_\mu} G \rightiso P$ of $P$ to $Q_\mu$ we may define $\gamma$ by
$$P \leftiso P_\mu \times^{Q_\mu} G \map{(p,g) \mapsto [g]} G/Q_\mu.$$
We leave it as a simple exercise for the reader to check this correspondence gives an equivalence.

\subsection{Filtered $G$-bundles with connection}
We next need to briefly discuss the algebraic notion of a flat connection on a $G$-bundle and the Griffiths transversality condition for a filtered $G$-bundle. 

\subsubsection{}
Let $X \rightarrow S$ be a family and $\Delta^2(1) = \Delta^2_{X/S}(1)$ a first order neighbourhood of the diagonal in $X \times_S X$. This comes equipped with a diagonal map $\delta: X \hookrightarrow \Delta^2(1)$ which is a closed immersion and two projection maps $p_1,p_2: \Delta^2(1) \rightarrow$. To express the notion of a flat connection we also need the first order neighbourhood $\Delta^3(1)$ of the diagonal in $X\times_S X \times_S X$ and its three projections $p_{12},p_{23},p_{13}:\Delta^3(1) \rightarrow \Delta^2(1)$.

With this notation, given $P \rightarrow X$ a $G_X$-bundle, a connection on $P$ (relative to $X/S$) is an isomorphism
$$\nabla: p_1^*P \rightiso p_2^*P$$
such that $\delta^*\nabla = \id_P$. We say such a connection is flat if it satisfies the cocycle condition
$$p_{13}^*(\nabla) = p_{23}^*(\nabla) \circ p_{12}^*(\nabla).$$

\subsubsection{}
\label{2.3.2}
Let $\vVec^\nabla_{X/S}$ denote the $\otimes$-category of vector bundles with flat connection on $X/S$, and suppose $G$ and $S$ are defined over a common Dedekind base $\Spec A$. Recall (by exactly the same argument as \cite[4.3.3]{tl}) that to give a $G_X$ bundle with flat connection is the same as giving a fibre functor
$$\Rep_A(G) \rightarrow \vVec^\nabla_{X/S}$$
or just giving a fibre functor $\omega: \Rep_A(G) \rightarrow \vVec_X$ and endowing for one's favourite faithful representation $G \hookrightarrow GL(V)$ the vector bundle $\omega(V)$ with a flat connection for $X/S$ that is trivial on the $G$-invariant subspaces of $\omega(V)^\otimes$.

\subsubsection{}
We next discuss the ``Griffiths transversality'' condition. Suppose $(\cV,\nabla,\Fil^\bullet)$ is a vector bundle with a connection on $X/S$ and decreasing filtration. The Griffiths condition states that for all $i$,
$$\nabla(\Fil^i) \subset \Fil^{i-1} \otimes \Omega_{X/S}^1.$$
Let $\Fil^\nabla_{X/S}$ denote the category of filtered vector bundles on $X$ with flat connection (relative to $X/S$) satisfying the Griffiths condition.

Given a $G$-bundle $P$ with a Rees structure and a flat connection, we get a fibre functor taking values in filtered vector bundles with flat connection, and we say it is a \emph{filtered $G$-bundle with flat connection} only if the image of this fibre functor lies in $\Fil^\nabla_{X/S}$. We denote the category of such bundles by $\Fil^{\nabla,G}_{X/S}$.

\begin{lem}
\label{2.3.4}
\begin{enumerate}
\item To check a $G$-bundle $P$ with Rees structure and flat connection lies in $\Fil^{\nabla,G}_{X/S}$ it suffices to check the Griffiths condition on $\omega_P(V)$ for $V$ a faithful representation of $G$.
\item Given $g:Y \rightarrow X$, the natural pullback defines a functor
$$g^*: \Fil^{\nabla,G}_{X/S} \rightarrow \Fil^{\nabla,G}_{Y/S}.$$
\item Given a map of connected reductive $A$-groups $G \rightarrow H$ the usual pushforward of torsors functor defines
$$(-) \times^G H: \Fil^{\nabla,G}_{X/S} \rightarrow \Fil^{\nabla,H}_{X/S}.$$
\end{enumerate}
\end{lem}
\begin{proof}
For (1) note that if $\omega_P(V)$ satisfies the Griffiths condition, so does $\omega_P(V)^\otimes$ and for any representation $W$ of $G$, $\omega(W)$ is a summand of $\omega_P(V)^\otimes$ so also satisfies the Griffiths condition. The other parts follow because (for (2)) the pullback of a vector bundle with connection satisfying Griffiths also satisfies Griffiths and (for (3)) by the diagram 

$$\omega_{P\times^G H}: \Rep_A(H) \map{\operatorname{Res}} \Rep_A(G) \map{\omega_P} \Fil^{\nabla}_{X/S}.$$
\end{proof}

\subsection{Filtered $F$-crystals with $G$-structure}
For this section, suppose $G/\Zp$ is connected reductive, $\kappa/\Fp$ a finite field, $W=W(\kappa)$ its ring of Witt vectors and $K=W[1/p]/\Qp$ the corresponding unramified field extension. Suppose we are working over a smooth scheme $X/W$ whose formal completion we denote by $\fX/\Spf(W)$ and rigid generic fibre by $\fX_{\eta}/\Spa(K,W)$.

\subsubsection{}
Recall that with such a setup we have a notion of a (coherent) \emph{crystal} on $\fX/\Spf(W)$ which is a finite locally free sheaf $\cF/\fX$ together with a flat topologically quasi-nilpotent connection $\nabla$ with respect to $\fX/W$. Letting $\operatorname{Crys}_{\fX/W}$ be the category of such, there is in particular a $p$-adic completion functor

$$\hat{(-)}: \vVec_{X/W}^\nabla \rightarrow \operatorname{Crys}_{\fX/W}.$$

One can also of course define the notion of a filtered crystal via the Griffiths condition as in the previous section.

Another important construction for us will be the category $\operatorname{Crys}_{\fX_{\eta}/K} = \operatorname{Crys}_{\fX/W}[1/p]$ of \emph{isocrystals} which can be viewed either as the $p$-isogeny category of $\operatorname{Crys}_{\fX/W}$ or as the category of vector bundles on $\fX_{\eta}$ with flat topologically quasi-nilpotent connection. From either perpsective it is clear there is a natural functor

$$(-)_{\eta}: \operatorname{Crys}_{\fX/W} \rightarrow \operatorname{Crys}_{\fX/W}[1/p].$$

Again of course a similar functor exists for the corresponding categories also equipped with a flat filtration satisfying the Griffiths condition, which we will refer to as ``filtered crystals'' and ``filtered isocrystals'' respectively with the flatness and Griffiths conditions implicit.

\subsubsection{}
We next recall the notion of a filtered $F$-crystal. Irritating technical considerations can arise with respect to Frobenius acting on lattices, so we will first define a filtered $F$-isocrystal and then effectively define filtered $F$-crystals to be strongly divisible lattices in $F$-isocrystals. Perhaps this will conflict with some definitions in the literature but it suffices for our purposes and seems to be that which is minimally confusing.

First, letting $F:X_0 \rightarrow X_0$ denote absolute Frobenius, we cover $\fX$ by affine opens $U_\alpha$ for which (as may be seen by taking local co-ordinates and using the smoothness of $X$) we have lifts $\sigma_\alpha: U_\alpha \rightarrow U_\alpha$ of $F$ which may be assumed finite, flat and \'etale in characteristic zero. A \emph{filtered $F$-isocrystal} is then a filtered isocrystal $\cF/\fX_{\eta}$ together with for each $\alpha$ a Frobenius isomorphism
$$\phi_\alpha: \sigma_\alpha^*\cF_{U_\alpha} \rightiso \cF_{U_\alpha}.$$
These are required to satisfy two conditions: they must be horizontal with respect to the natural connections on both sides. Second, they must be ``compatible on overlaps'' in the sense that the map 
$$\sigma_{\alpha}^*\cF|_{U_\alpha \cap U_\beta} \rightiso \cF|_{U_\alpha \cap U_\beta} \leftiso \sigma_{\beta}^*\cF|_{U_\alpha \cap U_\beta}$$
must be the natural isomorphism induced by the connection $\nabla$ (since $\sigma_\alpha, \sigma_\beta$ are identical modulo $p$). In particular, the latter condition ensures that this notion is independent of the choice $\{(U_\alpha,\sigma_\alpha)\}$ of local Frobenius lifts. 

We let $\FCrys_{\fX_{\eta}/K}$ denote the category of such filtered $F$-isocrystals. It is well-known that given a map $f: \fX_{\eta} \rightarrow \fY_{\eta}$ of rigid spaces over $K$, there is a canonical pullback functor
$$f^*: \FCrys_{\fY_\eta/K} \rightarrow \FCrys_{\fX_\eta/K}$$
which we emphasise does not depend on making choices of local Frobenius lift.

\subsubsection{}
Next as we head towards a definition of filtered $F$-crystals we first revisit the Griffiths condition, remaining in the context of a $p$-adic formal smooth scheme but exploring a direction that we suspect could be interesting in other situations.

If $U \subset \fX$ an open and $U \hookrightarrow T$ a nilpotent divided power thickening we may give $T$ the filtration by divided powers (by which we mean only the divided powers on $\Ker(\cO_T \rightarrow \cO_U)$, not including any in the ``$p$ direction'') and using the Rees construction obtain the a $\Gm$-equivariant formal scheme $\cT$ over $\A^1_{W}$. Moreover if we give the ring of functions $U$ the filtration putting everything in degree zero, the map $U \hookrightarrow T$ trivially preserves the filtrations which translates into a $\Gm$-equivariant map $U \times \A^1 \hookrightarrow \cT$.

Let $\cF$ be a crystal on $X_0/W$ and suppose $\cF_{\fX}$ comes with a flat filtration $\Fil^\bullet$. We say that the pair $(\cF,\Fil^\bullet)$ satisfies the Crystal-Filtration condition on $\fX/W$ if for any nilpotent divided power thickening $U \hookrightarrow T$ of an open $U \subset \fX$ and pair of sections $s_1,s_2: T \rightarrow U$, the natural isomorphism given by the crystal property
$$s_1^* \cF_U \rightiso \cF_T \rightiso s_2^*\cF_U$$
is an isomorphism of filtered modules with respect to the divided power filtration on $T$.

\begin{lem}
A pair $(\cF,\Fil^\bullet)$ consisting of a crystal and a filtration satisfies the Crystal-Filtration condition iff on $\fX$ the connection $\nabla: \cF_{\fX} \rightarrow \cF_{\fX} \otimes \Omega^1_{\fX/W}$ and filtration $\Fil^\bullet$ satisfy Griffiths transversality.
\end{lem}
\begin{proof}
We may check the claim in a formal neighbourhood of each closed point of $\fX$, so assume $\fX = \Spf W[[u_1,...,u_d]]$. Consider the thickening $$t:B_1=W[[u_1,...,u_d]][\epsilon_1,...,\epsilon_d]/(\epsilon_i^2 \ \forall i) \twoheadrightarrow W[[u_1,...,u_d]]$$
and $s_0,s_1:W[[u_1,...,u_d]] \rightarrow B_1$ defined by $s_0(u_i) = u_i, s_1(u_i) = u_i + \epsilon_i$ the two obvious sections. Recall that the connection is defined by the property that the map
$$B_1 \otimes_{s_0, W[[u_1,...,u_d]]} \cF \rightiso B_1\otimes_{s_1, W[[u_1,...,u_d]]} \cF$$ 
coming from the precrystal structure is (the $B_1$-linearisation of) $\id \oplus \nabla$ (identifying $\epsilon_i$ with $du_i$).

Using this setup let us prove the lemma. If $\cF$ is a filtered precrystal then since $(\epsilon_1,...,\epsilon_d) = \Fil^1(B_1)$, the condition that $1 \oplus \nabla$ respect the filtrations translates into the condition that
$$\nabla(\Fil^p) \subset \Fil^p((\epsilon_1,...,\epsilon_d) \cF) = (\epsilon_1,...,\epsilon_d) \Fil^{p-1}$$
which is the usual Griffiths transversality condition.

Conversely, if we have the Griffiths transversality condition, we recall the usual reconstruction of a crystal from a module with integrable connection. Firstly, integrality of the connection allows us to construct for $B_n = d.p.e(W[[u_1,...,u_d]][\epsilon_1,...,\epsilon_d]/(\epsilon_i^{n+1}\ \forall i))$ a canonical isomorphism
$$\theta_n: \cF \otimes_{W[[u_i]]} B_n \rightiso \cF \otimes_{W[[u_i]]} B_n.$$
One writes this down by writing $\nabla = \sum_i \epsilon_i \nabla_i$ where each $\nabla_i:\cF \rightarrow \cF$ is an endomorphism. Integrality of the connection implies the $\nabla_i$ commute, and one can define $\theta_n$ as the $B_n$-linearisation of 
$$f \mapsto \sum_{I} \nabla_I (f) \otimes \epsilon^{[I]}$$
(where $I$ varies over all multisets of size at most $n$ taking values in $\{1,...,d\}, \epsilon^{[I]} :=\prod_{k\in i} \epsilon_k^{[\mult_I(k)]}$, $\nabla_I = \prod \nabla_k^{\mult_I(k)}, \nabla_{\emptyset} = \id$ and $I! = \prod_{k \in i} \mult_I(k)!$).

Griffiths transversality implies $\nabla_i(\Fil^p) \subset \Fil^{p-1}$, which implies that whenever $|I|=m$, $\nabla_I(f) \in \Fil^{p-m}$, so the filtered precrystal condition is satisfied for $\theta_n$ because 
$$\nabla_I(\Fil^p) \otimes \epsilon^{[I]} \subset \nabla_I(\Fil^p) \otimes \Fil^m(B_n) \subset \Fil^{p-m} \cF \otimes \Fil^m(B_n) \subset \Fil^p(\cF \otimes_{W[[u_i]]} B_n).$$

Now recall that for a general nilpotent divided power thickening $B$ of $W[[u_i]]$, and $t_1,t_2:W[[u_i]] \rightarrow B$ a pair of sections, the isomorphism given by the crystal associated to the module $\cF$ with integrable connection $\theta: \cF \otimes_{t_1} B \rightiso \cF \otimes_{t_2} B$ one defines is, similarly to the case above,
$$f \mapsto \sum_{I} \nabla_I (f) \otimes (t_1(u)-t_2(u))^{[I]}$$
and just as above one sees that this respects the filtration.

\end{proof}

Since a section $T \rightarrow U$ will automatically preserve the filtration (noting that $\cO_T=\Fil^0 \cO_T$) $s_1$ and $s_2$ induce $\Gm$-equivariant sections $s_1,s_2:\cT \rightarrow U \times \A^1$. By the Rees equivalence we therefore obtain the following.

\begin{cor}
\label{griffiths cor}
To give a filtered crystal (i.e. flatly filtered, satisfying the Griffiths condition) on $\fX$ is to give a Rees bundle $\cF /\fX \times \A^1$ with a crystal structure on the fibre at $t=1$ which has the property that for every open $U \subset \fX$, divided power thickening $U \hookrightarrow T$ and pair of sections $s_1,s_2: T \rightarrow U$ (which induce natural sections $\cT \rightarrow U \times \A^1$) the natural isomorphism over $T$
$$s_1^*\cF_{U,1} \rightiso s_2^*\cF_{U,1}$$
coming from the crystal structure extends to a $\Gm$-equivariant isomorphism over $\cT$
$$s_1^*\cF \rightiso s_2^*\cF.$$
\end{cor}

\subsubsection{}
We like this characterisation because it gives a geometric construction of the category of strongly divisible $F$-crystals, as follows. Let $\cF$ be a filtered crystal, with $\Rees(\cF)$ its associated Rees module over $\fX \times \A^1$. It is easy to see that for each $\alpha$
$$\Rees(\sigma_\alpha^*\cF)/(t-p) = \sum_{i} p^{-i} \sigma_\alpha^*\Fil^i(\cF) \subset \sigma_\alpha^*\cF[1/p],$$
and by the above corollary this fibre comes equipped with a canonical flat connection (it is a fun exercise to show this agrees with the definition in Faltings \cite[p34]{fal1}). We define a \emph{filtered $F$-crystal} to be a filtered crystal $\cF$ together with for each $\alpha$ a horizontal isomorphism of bundles
$$\phi_\alpha: \Rees(\sigma_\alpha^*\cF)/(t-p) \rightiso \cF|_{U_\alpha}$$
and with the property that when we invert $p$ these form a filtered $F$-isocrystal.

We recall \cite[2.3]{fal1} which tells us that under certain ``Fontaine-Laffaille'' conditions this category is independent of the choice of local lifts of Frobenius. More precisely, this property holds if there exist $b>a$ such that $b-a<p$ and $\Fil^a = \cF, \Fil^{b+1} = 0$. In general, due to famous issues involving factors of $p$ in the denominators of Taylor expansions, one has no guarantee of this property, although in this paper our results would be unchanged if we demanded isomorphisms for \emph{all} local lifts of Frobenius and used that in the definition, so we do not dwell on these details.

We let $\FCrys_{\fX/W}$ denote the category of such, noting that it is an exact $\Zp$-linear rigid $\otimes$-category, and $\FCrys_{\fX/W}[1/p] := \FCrys_{\fX_{\eta}/K}$ the category of filtered $F$-isocrystals on the rigid generic fibre. Of course by definition there is a localisation functor
$$\FCrys_{\fX/W} \rightarrow \FCrys_{\fX/W}[1/p].$$

It will also be convenient to define a \emph{weak filtered $F$-crystal} to be a filtered $F$-isocrystal $\cE_{\eta}$ together with a filtered crystal $\cF$ and an identification $\cF[1/p] \rightiso \cE_{\eta}$. In other words, one with the strong divisibility condition relaxed. We obtain thus an enlargement $\FCrys_{\fX/W}^{weak}$ of $\FCrys_{\fX/W}$ which enjoys the benefit of being independent of the choice of local Frobenius lifts, and pullback functors along morphisms that do not need to intertwine some local Frobenius lifts. In applications, it is therefore often easy to produce a weak filtered $F$-crystal and subsequently show it is a true filtered $F$-crystal with respect to a choice of Frobenius lift.

\subsubsection{}
With these discussions in place, it is easy to define a filtered $F$-crystal with $G$-structure as a $\Zp$-linear fibre functor 
$$\omega: \Rep_{\Zp}(G) \rightarrow \FCrys_{\fX/W}$$
and we let $\FCrys_{\fX/W}^G$ denote the category of such.

Equivalently, such an object can be described explicitly as a $G$-torsor $\cP/\fX$ with flat topologically quasi-nilpotent connection together with a Rees structure $\cR(\cP)$ satisfying Griffiths transversality (a condition one can make geometric via (\ref{griffiths cor})) and isomorphisms

$$\phi_\alpha: \sigma_\alpha^*\cR(\cP)/(t-p) \rightiso \cP|_{U_\alpha}$$
that are horizontal (where the connection on the LHS is that induced via (\ref{griffiths cor})) and such that for each $\alpha,\beta$, the isomorphisms $\phi_\alpha$ and $\phi_\beta$ agree on restriction to $U_\alpha \cap U_\beta$ after inverting $p$. 

This equivalence follows easily from Broshi's Tannakian formalism and the following lemma.

\begin{lem}
\label{silly trick}
Given any fibre functor $\omega: \Rep_{\Zp}(G) \rightarrow \FCrys_{\fX/W}$, its composite with the forgetful functor
$\phi: \FCrys_{\fX/W} \rightarrow \operatorname{Fil-Crys}_{\fX/W}$ can be factored canonically
$$\Rep_{\Zp}(G) \map{\otimes_{\Zp}W} \Rep_{W}(G_{W}) \map{\omega'} \operatorname{Fil-Crys}_{\fX/W}$$
where $\omega'$ is a $W$-linear fibre functor.
\end{lem}
\begin{proof}
We use a trick from Deligne and Milne's paper on Tannakian categories. Recall that $\operatorname{Fil-Crys}_{\fX/W}$ is $W$-linear, and this is all we will need.

Given $\rho:G_W \rightarrow GL(V)$ with $V$ a finite free $W$-module, we can form $\Res(\rho): G \rightarrow \Res_{W/\Zp}(G_W) \rightarrow \Res_{W/\Zp}(GL_W(V)) \subset GL_{\Zp}(V)$. This is naturally a $W$-module object in the category $\Rep_{\Zp}(G)$. Applying $\phi \circ \omega$ we obtain a $W\otimes_{\Zp}W$-linear object $\cE(\Res(\rho))$ in $\operatorname{Fil-Crys}_{\fX/W}$. Finally we define using the map $W \otimes_{\Zp} W \ni w_1 \otimes w_2 \mapsto w_1w_2 \in W$
$$\omega'(\rho) := \cE(\Res(\rho))\otimes_{W \otimes_{\Zp} W} W.$$

As in \cite[3.10]{dm} this has the required property of factoring $\omega$.
\end{proof}

\subsubsection{}
One can also make a geometric definition of filtered $F$-isocrystals with $G$ structure as a $G$-torsor on $\fX_\eta$ with the obvious extra structures, denoted by $\FCrys_{\fX/W}^G[1/p]$ and again there is a localisation functor which coincides with taking the generic fibre.

\subsubsection{}
We also remark that one has obvious notions of the two functorial constructions that exist for torsors, with a caveat in (2) to account for the presence of noncanonical Frobenius lifts in the definition of a filtered $F$-crystal and the subsequent lack of robustness in its definition.
\begin{enumerate}
\item Given a $\Zp$-group morphism $G \rightarrow H$, there is a functor
$$\FCrys_{\fX/W}^G \rightarrow \FCrys_{\fX/W}^H$$
denoted by $\cP \mapsto \cP \times^G H$ and most easily defined using the fibre functor point of view as the composite
$$\Rep_\Zp H \map{\operatorname{Res}} \Rep_{\Zp} G \map{\omega_{\cP}} \FCrys_{\fX/W}.$$

\item Given a morphism of smooth formal $W$-schemes $g: \fX \rightarrow \fY$ compatible with choices of local Frobenius lift, we have a pullback functor
$$g^*: \FCrys_{\fY/W}^G \rightarrow \FCrys_{\fX/W}^{G}.$$
\end{enumerate}

Note that (2) requires that $g$ be compatible with local Frobenius lifts so that we may canonically identify $g^*\sigma_\alpha^*\cF = \sigma_\alpha^*g^*\cF$ and so make sense of pulling back the morphisms $\phi_\alpha$ along $g$ to get a map $$g^*(\phi_\alpha): \sigma_\alpha^*g^*\cF[1/p] \rightiso g^*\cF[1/p].$$ It also requires an argument checking that the lattices in the image which are a priori only in $\FCrys_{\fX/W}^{weak}$ are in fact strongly divisible. This may be done by using smoothness to take for each $\alpha$, and for $W'/W$ some finite \'etale extension, a Teichmuller point\footnote{Recall that for any mod $p$ point $x_0$ of a smooth $p$-adic formal scheme $\fX$ with Frobenius lift $\Phi: \fX \rightarrow \fX$, there is a unique lift $x$ of $x_0$ to $W(\kappa(x_0))$ such that the square
$$
\begin{CD}
   \Spf W(\kappa(x_0))  @>x>> \fX \\
@V\sigma VV        @V \Phi VV\\
  \Spf W(\kappa(x_0))   @>x>>  \fX
\end{CD}
$$
commutes, called the \emph{Teichmuller lift.} We call any such point a Teichmuller point.} $x: \Spf W' \rightarrow U_\alpha$ for each component of $U_\alpha$. It is clear that for any representation $V$ of $G$ and $\omega \in \FCrys_{\fY/W}^G$, $x^*g^*\omega(V)$ is a strongly divisible lattice. But now by (\ref{sd criterion}) we deduce that $g^*\omega(V)$ is strongly divisible.

\subsection{Lattices and $\fS$-modules}
We will sometimes wish to prove statements about (strongly divisible) filtered $F$-crystals by pulling them back to an unramified point and relating them to Kisin's theory of $\fS$-modules. We therefore digress to summarize this theory and the necessary comparison result, assembling what we need from the literature.

\subsubsection{}
Let $\fS=W[[u]]$, with its Frobenius map $\phi:\fS \ni u \mapsto u^p \in \fS$. Let $E(u)=u-p$, and let $\Mod_{/\fS}^\phi$ denote the category of finite free $\fS$-modules $\fM$ equipped with a semilinear Frobenius
$$\phi: \phi^*(\fM)[1/E(u)] \rightiso \fM[1/E(u)].$$
We also need the notation $\Mod_{\fS}^{\phi, r}$ for the subcategory of such modules where $1\otimes \phi(\phi^*(\fM)) \subset \fM$ and the cokernel is killed by $E(u)^r$.

Let us make the additional notation $\Rep_{\Zp}^{crys}(\Gamma_K)$ for the category of crystalline representations of $\Gamma_K$ on finite free $\Zp$-modules. Recall the following result from \cite[1.2.1]{kis2} following \cite{kis1}.

\begin{thm}[Kisin]
\label{kismod}
There is a fully faithful tensor functor
$$\fM: \Rep_{\Zp}^{crys}(\Gamma_K) \hookrightarrow \Mod_{/\fS}^\phi$$
compatible with the formation of symmetric and exterior powers and unramified base change, and equipped with a canonical isomorphism for each $L \in \Rep_{\Zp}^{crys}(\Gamma_K)$
$$D_{crys}(L[1/p]) \cong (\fM(L)/u) [1/p].$$
\end{thm}

\subsubsection{}
Now let $\FCrys^{[0,p-2]}_{\Spf{W}/W} \subset \FCrys_{\Spf{W}/W}$ be the subcategory of those filtered $F$-crystals\footnote{Note that with our definitions the ``strongly divisible'' condition is implicit always.} $\cE$ with $\Fil^0(\cE) = \cE$ and $\Fil^{p-1}(\cE)=0$, and let $\Rep_{\Zp}(\Gamma_K)$ be the category of continuous representations of $\Gamma_K := \Gal(\bar{K}/K)$ on finite free $\Zp$ lattices. On the one hand, the theory of Fontaine-Laffaille gives a fully faithful functor
$$\bL: \FCrys^{[0,p-2]}_{\Spf{W}/W} \hookrightarrow \Rep^{crys}_{\Zp}(\Gamma_K)$$
whose image is precisely the lattices in crystalline representations with Hodge-Tate weights in the range $[0,p-2]$.

Let us fix $L \in \Rep_{\Zp}^{crys}(\Gamma_K)$ a lattice in a crystalline representation with Hodge Tate weights in the range $[0,p-2]$. By classical $p$-adic Hodge theory we know that $\D^{FL} := \bL^{-1}(L)$ is canonically a lattice in $D_{crys}(L\otimes \Qp)$. On the other hand by the theory of $\fS$-modules we may construct another lattice $\D^\fS := \phi^*\fM(L)/u \subset D_{crys}(L \otimes \Qp)$. The following result is probably known to the experts but since we could find no clear such statement in the literature we assemble it here.

\begin{prop}
\label{FL equality}
With the setup above, the lattices $\D^{FL},\D^{\fS} \subset D_{crys}(L \otimes \Qp)$ are equal.
\end{prop}

\begin{proof}
Following Breuil, Kisin and Liu we may consider the following functors. First \cite[2.2.2 (2)]{bre}
$$\cM_W: \FCrys^{[0,p-2]}_{\Spf{W}/W} \ni \D \mapsto \D\otimes_W S \in \Mod^{\phi,N}_{/S},$$
second \cite[3.3]{liu1} 
$$\cM_{\fS}: \Mod_{/\fS}^{\phi, p-2} \ni \fM \mapsto \phi^*\fM \otimes_{\fS} S \in \Mod^{\phi,N}_{/S}.$$

It is evident from their description that we have a canonical identification $$\D \cong \cM_W(\D)/u.$$

Moreover, by the triangle in the top row of \cite[p17]{liu1} we see that whenever the $\fS$-module $\fM$ and the $S$-module $\cM$ are associated to the same $L \in \Rep_{\Zp}(\Gamma_{K^\infty})$ that they are identified under $\cM_{\fS}$. In our situation (and by the definition of ``associated'' we inherit from Liu's paper which is in particular compatible with that of Fontaine-Laffaille) we have by construction that $\cM_W(\D^{FL})$ and $\fM=\fM(L)$ are both associated to $L|_{\Gamma_{K_\infty}}$.

We conclude that $\cM(\fM) = \cM(\D^{FL})$ and so in particular
$$\D^\fS := \phi^*\fM/u = \cM(\fM)/u = \cM(\D^{FL})/u = \D^{FL},$$
as required.
\end{proof}

This result will be particularly useful in combination with the following fairly well-known lemma, for which we supplied a proof in \cite[4.1.3]{tl}.

\begin{lem}
\label{lattices on points}
Let $R$ be a Dedekind domain with fraction field $K$ and $X/R$ an integral scheme. Let $\cV/X_K$ be a vector bundle, and $\cL,\cL'$ two lattices in $\cV$: i.e. vector bundles on $X$ extending $\cV$. Suppose for some $R'/R$ finite we can find a point $s: \Spec R' \rightarrow X$ such that $s^*\cL = s^*\cL'$ as lattices in $s^*\cV$. Then it follows that $\cL = \cL'.$
\end{lem}

Another handy corollary of the above is the following.

\begin{cor}[Strong divisibility criterion]
\label{sd criterion}
Suppose $(\cF,\cE_{\eta})$ is a weak filtered $F$-crystal on $\fX=\hat{X}$ with $X/W$ smooth, $(U_\alpha,\sigma_\alpha)$ a set of local Frobenius lifts with each $U_\alpha$ integral and $x_\alpha:\Spf W_\alpha \hookrightarrow U_\alpha$ Teichmuller points of $\fX$. Then if each $x_\alpha^*(\cF,\cE_{\eta})$ in fact lies in $\FCrys_{\Spf W_\alpha/\Spf W_\alpha}$, it follows that $(\cF,\cE_{\eta})$ lies in $\FCrys_{\fX/W}$.
\end{cor}
\begin{proof}
Recall that a weak $F$-crystal $(\cF, \cE_\eta, \iota:\cF[1/p] \rightiso \cE_\eta)$ lies in $\FCrys_{\fX/W}$ if for the Frobenius lifts $(U_\alpha,\sigma_\alpha)$ we have the that the lattices $\phi_\alpha(\iota(\Rees(\sigma_\alpha^*\cF|_{U_\alpha})/(t-p)))$ and $\iota \cF|_{U_\alpha}$ inside $\cE_{\eta}|_{U_\alpha[1/p]}$ agree. The lemma tells us such equality can be checked after restriction to $x_\alpha$, whence the result.
\end{proof}

\subsection{Some $p$-adic comparison results with $G$-structure}
For our application to Galois representations we will need adapted versions and some properties of the $p$-adic comparison theorems of \cite{fal1} between \'etale and crystalline cohomology that hold in the setting with $G$-structures. The technical details involved in setting up the comparison theorems the way Faltings does are fairly heroic. Since we only give a summary they might also seem unmotivated, but the shape of the results themselves is easily digestible, so many readers may wish to just skip to the result statements.

\subsubsection{}
First, we recall the relative version of Fontaine's crystalline period rings following Faltings \cite{fal1}. Let $\Spec R \subset X$ be a small affine open in the sense that it is integral and there is an \'etale map $W[t_1^\pm, \dots t_d^\pm] \rightarrow R$. Form $\bar{R}/R$ in the case when $R\otimes_{W} \bar{K}$ is integral as the normalisation of $R$ in the maximal extension of the fraction field of this ring such that the normalisation of $R[1/p]$ in this field is unramified over $R[1/p]$. In the general case, $R\otimes_{W} \bar{K}$ is a product of integral domains and define $\bar{R}$ to be the product of rings obtained from each factor as above. Let $W(\bar{R}^\flat)$ the Witt vectors of the tilt, and $D(R)=\operatorname{DPE}(\theta:W(\bar{R}^\flat) \rightarrow \hat{\bar{R}})$ the divided power envelope. Finally, $A_{crys}(R)$ is the completion of $D(R)$ for the topology defined by the divided power ideals $(I+(p))^{[n]}$. Inverting $p$ we get $B_{crys}^+(R)$ and inverting a period for the cyclotomic character $t=\log[\epsilon]$ we get $B_{crys}(R)$.

By functoriality this has a continuous action of the profinite group $\Gamma_R = \Gal(\bar{R}/R)$. The Frobenius on $W(\bar{R}^\flat)$ induces a canonical Frobenius $F: D(R) \rightarrow D(R)$ extending over the larger rings. Faltings also gives it the `strong divisibility' filtration $\Fil^n D(R) = I^{[n]} \cap \Ker(D(R) \map{F} D(R) \twoheadrightarrow D(R)/p^nD(R))$, closed and inducing the topology $A_{crys}(R) \rightiso \lim A_{crys}(R)/\Fil^n$. This object should be thought of as a huge local system containing all the ``crystalline \'etale local systems'' over $\Spec R$. Let $\hat{R}$ be the integral closure of the $p$-adic completion of $R$, and one can repeat the construction and obtain slightly different rings $A_{crys}(\hat{R}), B_{crys}(\hat{R})$.

Recall that a lisse $\Zp$-sheaf on $\Spec R[1/p]$ is determined by the data of a representation $r: \Gamma_R \rightarrow GL_{\Zp}(L)$, where $L$ is a finite free $\Zp$-module. Given such a datum, we can form the continuous $\Gamma_R$-module $$\mathbb{M}_{et}(L)_R := L \otimes_{\Zp} A_{crys}(\hat{R}),$$ which also inherits a filtration and a Frobenius from the second factor.

On the other hand given a filtered virtual $F$-crystal $\cE$ on $\Spf(\hat{R}) / W)$, we can evaluate it on $B_{crys}(\hat{R})$ as follows.\footnote{Since $\bar{R}$ is extremely non-smooth over $W(k)$, this procedure is rather subtle.} There is a natural map $\hat{R} \rightarrow \hat{\bar{\hat{R}}} \cong A_{crys}(\hat{R})/\Fil^1.$ By smoothness, we can lift to a map $g: \hat{R} \rightarrow A_{crys}(\hat{R})$, and the connection satisfying Griffiths transversality on $\cE$ gives a well-defined module
$$\mathbb{M}_{crys}(\cE)_R := \cE(\hat{R}) \otimes_{\hat{R}} A_{crys}(\hat{R})$$
equipped with a filtration. 

There is a naive notion of a Frobenius and Galois action on each factor in this tensor product, but since our arbitrary lift $g$ need not intertwine either structure we need to be a little careful. It turns out that if we fix a Frobenius lift on $\hat{R}$ that is \'etale in characteristic 0 (which is always possible if $R$ is small: just raise all co-ordinates to their $p$th power), we can produce a $F$-equivariant $g$, using the recipe in \cite[p36]{fal1}. Provided we use such a Frobenius in our calculations, the naive formula $\phi(m\otimes b) = \phi(m) \otimes \phi(b)$ will suffice and is invariant under the connection.

Unfortunately, we cannot also make $g$ Galois equivariant, but we can make precise the reason why and following Faltings \cite[5.5]{fal1}, we equip $\mathbb{M}_{crys}(\cE)_R$ with the ``horizontal'' Galois action given by, for $\sigma \in \Gamma_R$,
$$\sigma(m \otimes b) = \sum_{I} \nabla(\partial)^I(m) \otimes \frac{\beta(\sigma)^I \sigma(b)}{I!},$$
where $\beta: \Gal(R[\mu_{p^\infty}]/R) \cong \lim \mu_{p^n} \rightarrow \Fil^1(B_{crys})$ is the standard canonical map $\epsilon=(\zeta_1,\zeta_2,\dots) \mapsto \log [\epsilon]$ embedding the cyclotomic periods in $B_{crys}$. We remark that on horizontal sections the naive formula is recovered. Both of these constructions make essential use of the technique of passing to a  ``preperfectoid'' extension $R_\infty/R$ (more precisely, one forms the `completed perfection' of $R$ using the Frobenius lift on $R$), from which there is a canonical map $R_\infty \rightiso W(R_\infty/pR_\infty) \rightarrow A_{crys}(R)$.

\subsubsection{}
Let $\FCrys^{[0,p-2]}_{\fX/W} \subset \FCrys_{X/W}$ be the subcategory of those filtered $F$-crystals $\cE$ with $\Fil^0(\cE) = \cE$ and $\Fil^{p-1}(\cE)=0$, and let $\Lisse_{\Zp}(X_K)$ be the category of lisse $\Zp$-sheaves on $X_K$. 

Following Faltings (and Fontaine-Laffaille), we can define a functor
$$\bL: \FCrys^{[0,p-2]}_{\fX/W} \rightarrow \Lisse_{\Zp}(X_K)$$

by the recipe (for $\Spec R \subset X$ a small affine open faithfully flat over $W$)
$$\bL(\cE)|_{\Spec R} := \lim_r (\Hom_{A_{crys}(\hat{R}), \Fil, \phi}(\bM_{crys}(\cE)/p^r, A_{crys}(\hat{R})[1/p]/A_{crys}(\hat{R}))^*),$$
where the final $(-)^*$ denotes Pontryagin duality as in \cite[p43]{fal1} and (we remark) makes the functor covariant. By \cite[2.3, 2.6]{fal1} these definitions make sense and give rise to a full subcategory which we will call the \emph{Fontaine-Laffaille} $\Zp$-sheaves, and one should think of as globalised lattices in crystalline representations with Hodge tate weights restricted to the range $[0,p-2]$.

They enjoy the following comparison result \cite[\S2 (h)]{fal1}.

\begin{thm}
There is a canonical functorial morphism of \'etale sheaves respecting $\phi$ and $\Fil$ which has an inverse up to $\beta^{\otimes (p-2)}$ 
$$\bM_{crys}(\cE) \rightarrow \bM_{et}(\bL(\cE)).$$
\end{thm}

\subsubsection{}
Recall \cite[p67]{fal1} that in general a filtered $F$-isocrystal $\cF$ and a lisse $\Qp$-sheaf $\mathbb{V}$ are \emph{associated} if we have isomorphisms functorial in $R$ 
$$\cF(B_{crys}(\hat{R})) \cong \mathbb{V} \otimes_{\Qp} B_{crys}(\hat{R})$$
that preserve Galois actions, Frobenius and filtration.
Note that the theorem above implies in particular that when $\cE \in \FCrys_{\fX/W}^{[0,p-2]}$, $\cE[1/p]$ and $\bL(\cE)[1/p]$ are associated. The following comparison result will be very useful in what follows.

\begin{prop}
\label{2.6.5}
Fix $X/W$ an integral scheme with an unramified point $x: \Spec W(k') \rightarrow X$. Suppose $G/\Zp$ is a connected reductive group and that we are given a lisse $\Zp$ sheaf with $G$-structure $$\omega_{et}:\Rep_{\Zp}(G) \rightarrow \Lisse_{\Zp}(X_K)$$ and a filtered $F$-crystal with $G$-structure
$$\omega_{crys}:\Rep_{\Zp}(G) \rightarrow \FCrys_{\fX/W}.$$
Suppose that for some faithful representation $G \hookrightarrow GL(V)$, $\omega_{crys}(V) \in \FCrys_{\fX/W}^{[0,p-2]}$ and we have an identification of lisse $\Zp$-sheaves $$\bL(\omega_{crys}(V)) \cong \omega_{et}(V).$$
Suppose furthermore that there are tensors $s_\alpha:1 \rightarrow V^\otimes$ such that $G = GL_{s_\alpha}(V)$ and $\omega_{crys}(s_\alpha)$ and $\omega_{et}(s_\alpha)$ are identified under the canonical functorial system of isomorphisms
$$\omega_{crys}(V)[1/p]^\otimes \otimes B_{crys}(\hat{R}) \cong \omega_{et}(V)[1/p]^\otimes \otimes B_{crys}(\hat{R}).$$
Then:
\begin{enumerate}
\item For any representation $G \rightarrow GL(V')$ we have that $\omega_{crys}(V')[1/p]$ and $\omega_{et}(V')[1/p]$ are associated.
\item Whenever $\omega_{crys}(V') \in \FCrys_{\fX/W}^{[0,p-2]}$, we have that $\bL(\omega_{crys}(V')) = \omega_{et}(V')$.
\end{enumerate}

\end{prop}
\begin{proof}
For (1), the inputs allow us to canonically identify the torsors 
$$\uIsom_{s_\alpha}(V_W, \omega_{crys}(V)) \otimes B_{crys}(\hat{R}) \cong \uIsom_{s_\alpha}(V, \omega_{et}(V))\otimes B_{crys}(\hat{R})$$
with their structures as filtered $F$-crystals with $G$-structure and as \'etale sheaves. The result then follows immediately from the relationship between torsors and fibre functors.

For (2), since (1) already implies that $\bL(\omega_{crys}(V'))[1/p] = \omega_{et}(V')[1/p]$ it remains to check that under this identification the two $\Zp$-lattices $\bL(\omega_{crys}(V'))$ and $\omega_{et}(V')$ agree. By \cite[2.6]{fal1} $\bL$ has an inverse and it suffices to check that the lattices $\omega_{crys}(V')$ and $\bL^{-1}(\omega_{et}(V'))$ inside $\omega_{crys}(V')[1/p]$ agree. But these are lattices in a vector bundle on an integral scheme so by (\ref{lattices on points}) it suffices to check equality at a point, and in particular we may use the unramified point $x$.

To check this, we use the facts about $\fS$-modules from the previous section. Note that we already know that $x^*\omega_{et}(V')[1/p]$ is crystalline by part (1) and $$x^*\omega_{crys}(V')[1/p] = D_{crys}(x^*\omega_{et}(V')[1/p])=:D'.$$
Inside this, we are considering two lattices: let $L:=x^*\omega_{et}(V')$. One the one hand we have $\D^{FL} := \bL^{-1}(L) \subset D'$. On the other hand we have $x^*\omega_{crys}(V')$ which we shall describe as a lattice using $\fS$-modules.

It is harmless (say, by making a finite base change if necessary) to assume $W=W(k')$. The idea is to use $\fS$-modules to describe a new fibre functor $\eta: \Rep_{\Zp}(G) \rightarrow \Mod_W$ equipped with an identification $\eta[1/p] \cong D_{crys} \circ \omega_{et}[1/p]$ of functors into $\Mod_K$  and check that as a functor into lattices it agrees with $x^*\omega_{crys}(V')$. The construction is
$$\eta: \Rep_{\Zp}(G) \map{\omega_{et}} \Rep_{\Zp}(\Gamma_{K}) \map{\fM} \Mod_{/\fS}^\phi \rightarrow \Mod_W$$
where the final arrow is $\fM \mapsto (\phi^*\fM)/u$. This functor comes with a canonical $\eta[1/p] \cong D_{crys} \circ \omega_{et}[1/p]$ by (\ref{kismod}).

What is more, by (\ref{FL equality}) we have the equality of lattices $\eta(V') = \D^{FL} \subset D'$, and will therefore be done if we can show that $x^*\omega_{crys}(V') = \eta(V')$. To do this, first note that $x^*\omega_{crys}(V) = \eta(V)$ again by (\ref{FL equality}) and so since $x^*\omega_{crys}$ is a tensor functor, if $V^T$ is any finite rank tensorial construction that contains $V'$ we also have
$$x^*\omega_{crys}(V^T) = x^*\omega_{crys}(V)^T = \eta(V)^T = \eta(V^T).$$
Finally, letting $e:V^T \rightarrow V'$ be the idempotent projector onto $V'$ and $\hat{e}$ its realisation under the identified functors
$$\hat{e} = x^*\omega_{crys}(e)[1/p] = D_{crys}\circ \omega_{et}[1/p](e) = \eta(e)[1/p]$$
we have that 
$$x^*\omega_{crys}(V') = \hat{e} x^*\omega_{crys}(V^T) = \hat{e} \eta(V^T) = \eta(V') \subset D',$$
finishing the proof.

\end{proof}

\subsubsection{}
\label{integrally associated}
The above is an example of the following defined situation. Let $X/W$ be a scheme, and consider a pair of fibre functors
$$\omega_{et}: \Rep_{\Zp}(G) \rightarrow \Lisse_{\Zp}(X_K)$$
and
$$\omega_{crys}: \Rep_{\Zp}(G) \rightarrow \FCrys_{\fX/W}.$$
We say that $\omega_{et}$ and $\omega_{crys}$ are \emph{integrally associated} if:
\begin{enumerate}
\item The \emph{fibre functors} $\omega_{et}[1/p]$ and $\omega_{crys}[1/p]$ are associated. In other words if we have isomorphisms functorial in $R$ between $\bM_{crys,R} \circ \omega_{crys}[1/p]$ and $\bM_{et,R} \circ \omega_{et}[1/p]$ that preserve Galois actions, Frobenius and filtration.
\item Whenever $V' \in \Rep_{\Zp}(G)$ is such that $\omega_{crys}(V') \in \FCrys^{[0,p-2]}_{\fX/W}$, we have that $\bL(\omega_{crys}(V')) = \omega_{et}(V')$. 
\end{enumerate}

We remark that this is a weaker condition that one would ideally like, with the integrality statement given only in the Fontaine-Laffaille range. In particular for $p=2$ the statement becomes almost vacuous in our context. However, we anticipate that once it can be fortified using a more robust version of relative integral $p$-adic Hodge theory than is currently available, our results will still hold.

\subsubsection{}
Finally, we recall some $p$-adic comparison theorems we will need from \cite{fal1}. First with $p$ inverted, we have the following special case of \cite[6.3]{fal1}.

\begin{thm}
\label{comp Qp}
Let $f:X \rightarrow Y$ be a proper morphism of smooth schemes over $W$, $\cE$ a filtered $F$-isocrystal on $\fX/W$, $\bL$ a lisse $\Qp$ sheaf on $X_K$ and assume $\cE$ and $\bL$ are associated. Then their direct images $R^if_*\cE$ and $R^if_*\bL$ (in the crystalline and \'etale categories respectively) are associated.
\end{thm}

Next, integrally, we have a more restrictive theorem \cite[6.2]{fal1} (including `\emph{Remark}').
\begin{thm}
\label{comp Zp}
Let $f:X \rightarrow Y$ be a proper morphism of smooth schemes over $W$, $a,b \geq 0$, $\cE \in \FCrys_{\fX/W}^{[0,a]}$. Then if $a+b < p-2$ or if $a+b=p-2$ and $f$ is an abelian scheme, then $R^bf_*\cE \in \FCrys_{\mathfrak{Y}/W}^{[0,a+b]}$ and we have a natural identification of lisse $\Zp$-sheaves
$$\bL(R^bf_*\cE) = R^bf_*\bL(\cE).$$
\end{thm}

\section{Filtered $F$-crystals on Shimura Varieties}
\label{S3}

\subsection{Canonical $G$-bundles on integral models of Shimura Varieties}
In this section we revisit the constructions of \cite[\S4]{tl}, using $\fS$-modules to define canonical lattices inside the ``standard principal bundles'' on integral models Shimura varieties, but in the rather less technical setting where our integral models are over $\cO_{E,(v)}$ rather than (as in \cite{tl}) over the whole of $\cO_E[1/N]$. It will be obvious that the definitions in both papers are compatible. Using local $p$-adic Hodge theoretic methods including the recent paper of Liu and Zhu \cite{lz} we may avoid the extra condition imposed for example in \cite{milne3} and then \cite{tl} that $Z(G)^\circ$ is split by a CM field.

\subsubsection{}
Our setup will be as follows. Suppose $G/\Z_{(p)}$ is a (connected) reductive group (which we will freely confuse with $G_\Q$) and $(G,\fX)$ a Shimura datum in the usual sense, for example of \cite[\S 2.1]{del2}. Consider, for $K_p = G(\Zp)$ and $K^p \subset G(\A^{\infty,p})$ compact open and varying in an inverse system, the tower of (canonical models for) Shimura varieties, each of which is a quasiprojective variety over the number field $E=E(G,\fX)$
$$\Sh_{K_p}(G,\fX) = \lim_{K^p} \Sh_{K^pK_p}(G,\fX).$$

This tower comes equipped with a continuous right `Hecke' action of $G(\A^{\infty,p})$ and for each $K^p$ we may identify the Hecke quotient
$$\Sh_{K_p}(G,\fX)/K^p = \Sh_{K^pK_p}(G,\fX).$$

\subsubsection{}
Recall that for $v|p$ a place of $E$, Kisin \cite{kis2} following Milne defines an \emph{integral canonical model} $\cS_{K_p}$ for $\Sh_{K_p}(G,\fX)$ to be an inverse limit $\cS_{K_p} = \lim_{K^p} \cS_{K^pK_p}$ with a $G(\A^{\infty,p})$-action such that:
\begin{enumerate}
\item Each $\cS_{K^pK_p}$ is a smooth quasiprojective $\cO_{(v)}$-scheme and is given together with an identification over $E$ with $\Sh_{K^pK_p}$ in such a way that extends to a resulting $G(\A^{\infty,p})$-equivariant identification
$$\cS_{K_p} \otimes E \rightiso \Sh_{K_p}.$$
\item The scheme $\cS_{K_p}$ has the \emph{extension property}: whenever $T/\cO_{(v)}$ is a regular formally smooth scheme, any map $T_E \rightarrow \cS_{K_p, E}$ must extend to a map $T \rightarrow \cS_{K_p}$.
\end{enumerate}
It is easy to show that these properties imply that such a model $\cS_{K_p}$ if it exists is unique up to canonical isomorphism, and the main theorem of \cite{kis2} is that such models exist whenever $(G,\fX)$ is of abelian type.

In this paper we will work locally, so from now on we will abuse notation slightly and assume we are working over a fixed prime $v|p$ of the reflex field $E$ and considering formal schemes $\cS_{K_p}/\cO_v$ whose generic fibre $\Sh_{K_p}/E_v$ we view as an adic space over the completion of $E$ at $v$. However, if we assume $Z(G)^\circ$ is split over a CM field, most of what we do works more globally (see \cite{tl}) and we would expect this to be true in general.

\subsubsection{}
\label{3.1.3}
Let $Z_{nc} \subset Z(G)$ be the largest subtorus of $Z(G)$ that is split over $\R$ but has no subtorus split over $\Q$, and set $G^c = G/Z_{nc}$. Recall \cite[3.1.3]{tl} that although one does not have functoriality of this construction in general if $(G,\fX) \rightarrow (G_2,\fX_2)$ is a map of Shimura data then there is an induced map $G^c \rightarrow G_2^c$.

Recall that Liu and Zhu \cite{lz} construct a filtered $G^c$-bundle with flat connection $P_{K_p}$ on the adic space $\Sh_{K_p}(G,\fX)$ as follows. First, they use the pro-Galois $G^c(\Zp)$-cover
$$\Sh(G,\fX) \rightarrow \Sh_{K_p}(\fX)$$
to construct a $\Qp$-linear fibre functor
$$\omega_{et}: \Rep_{\Qp}(G^c) \rightarrow \Lisse_{\Qp}(\Sh_{K_p}).$$
They then note using special points and their main theorem \cite{lz} that the image of this functor is de Rham, and that it therefore may be extended to a functor
$$\omega_{dR}: \Rep_{\Qp}(G^c) \rightarrow \Fil^\nabla_{\Sh_{K_p}/E_v}.$$
By the same argument of Deligne-Milne as in (\ref{silly trick}) this in particular gives a functor
$$\Rep_{E_v}(G^c_{E_v}) \rightarrow \Fil^\nabla_{\Sh_{K_p}/E_v}$$
which precisely defines a filtered $G^c$ bundle $P_{K_p}$ with flat connection on $\Sh_{K_p}/E_v$.

We remark that Liu and Zhu conjecture \cite[Remark 4.1 (ii)]{lz} that this should agree with the analytification of Milne's construction \cite[\S3]{milne3} in the case where $Z(G)^\circ$ is split by a CM field. One consequence of our argument (comparing the construction in this paper with that of \cite{tl}) is that this is certainly true in the abelian type case.

One nice consequence of this construction is that it is completely functorial \cite[3.9 (ii)]{lz} and extends $p$-adic Hodge theory over a point \cite[1.5 (i)]{lz}, giving the following compatibilities.
\begin{lem}
\label{lz facts}
\begin{enumerate}
\item Suppose $(G,\fX) \rightarrow (G_2,\fX_2)$ is a morphism of Shimura data induced by a map $G \rightarrow G_2$ of reductive models over $\Z_{(p)}$, and take $E \supset E(G,\fX),E(G_2,\fX_2)$, $v|p$ a place of $E$. Then pulling back $P_{G_2(\Zp)}$ along $$i: \Sh_{G(\Zp)}(G,\fX)_{E_v} \rightarrow \Sh_{G_2(\Zp)}(G_2,\fX_2)_{E_v}$$ we have a canonical identification
$$i^*P_{G_2(\Zp), E_v} \cong P_{G(\Zp), E_v} \times^{G^c} G_2^c$$
of filtered $G_2^c$-bundles with connection on $\Sh_{G(\Zp)}(G,\fX)_{E_v}$.
\item The Hecke operators $G(\A^{\infty,p})$ act equivariantly on $P_{G(\Zp)} \rightarrow \Sh_{G(\Zp)}(G, \fX)$.
\item Let $x \in \Sh_{K_p}(G,\fX)$ be any (closed) point. Then $x^*\omega_{dR} = D_{dR} \circ x^*\omega_{et}$ where $D_{dR}$ is Fontaine's functor $V \mapsto (V \otimes_{\Qp} B_{dR})^{\Gamma_{\kappa(x)}}.$
\end{enumerate}
\end{lem}

\subsubsection{}
\label{3.1.5}
With this setup in place we can make our first key definition. Suppose we have a (smooth) integral canonical model $\cS_{G(\Zp)}$ for $\Sh_{G(\Zp)}(G,\fX)$ (viewed as adic spaces over $\cO_v$ and $E_v$ respectively which we recall are absolutely unramified). A \emph{weak crystalline canonical model} $\cP_{G(\Zp)}$ for $P_{G(\Zp)}$ is a $G(\Afp)$-equivariant weak filtered $F$-crystal with $G^c$-structure $\cP_{G(\Zp)} \in \FCrys_{\cS_{G(\Zp)}/\cO_v}^{G^c}$ together with a $G^c_{E_v} \times G(\Afp)$-equivariant identification $i: \cP_{G(\Zp),E_v} \rightiso P_{G(\Zp)}$ compatible with connections and filtrations and satisfying the following ``crystalline points lattice + Frobenius'' condition.

Recall that the pro-Galois cover $\Sh(G,\fX) \rightarrow \Sh_{G(\Zp)}(G,\fX)$ can be used to define a $\Zp$-linear tensor functor
$$\omega_{et}: \Rep_{\Zp}(G^c) \rightarrow \Lisse_{\Zp}(\Sh_{G(\Zp)(G,\fX)}),$$
and we let
$$\omega_{crys}: \Rep_{\Zp}(G^c) \rightarrow \FCrys_{\cS_{G(\Zp)}/\cO_v}$$
be the $\Zp$-linear fibre functor induced by $\cP_{G(\Zp)}$.

Let $x \in \cS_{G(\Zp)}(W(k'))$ be a point defined over an unramified extension $K=W(k')[1/p]/E_v$ extending over the special fibre and for which $x^*\omega_{et}[1/p]$ takes values in crystalline representations of $\Gamma_{K}$ (from now on we call such points ``crystalline points''). Then we note that by (\ref{lz facts} (3)) we may identify
$$\theta_x: x^*i^*\omega_{crys} \cong D_{crys} \circ x^*\omega_{et}[1/p].$$

The ``crystalline points lattice + Frobenius'' (CPLF) condition says that for all crystalline points $x$, we have the following identifications.
\begin{itemize}
\item (L condition) The lattices on the left hand side defined by $L \mapsto x^*\omega_{crys}(L)$ agree under $\theta_x$ with the lattices on the right hand side defined by $L \mapsto \phi^*\fM(x^*\omega_{et}(L))/u$ and (\ref{kismod}).
\item (F condition) The Frobenius on $x^*i^*\omega_{crys}$ induced from that given on $i^*\cP_{G(\Zp)}$ viewed as an $F$-isocrystal with $G^c$-structure agrees under $\theta_x$ with the Frobenius on $D_{crys} \circ x^*\omega_{et}[1/p]$ coming from $p$-adic Hodge theory.
\end{itemize} 

Finally, a (strong) \emph{crystalline canonical model} $\cP_{G(\Zp)}$ must satisfy the additional condition that for all local Frobenius lifts $(U_\beta,\phi_\beta)$ on $\cS_{G(\Zp)K}$ at any finite level $K \subset G(\Afp)$ that $\cP_{G(\Zp)}/K$ is a (strongly divisible) filtered $F$-crystal with $G$-structure.

We can also admit a notion of crystalline canonical models over a finite \'etale extension of $\cO_v$ in the obvious fashion.

\begin{prop}
\label{main uniqueness}
\begin{enumerate}
\item The crystalline canonical model $(\cP_{G(\Zp)}, i)$ if it exists is unique up to canonical isomorphism.
\item Let $f: (G,\fX) \rightarrow (G_2,\fX_2)$ be a map of Shimura data induced by a map $G \rightarrow G_2$ of reductive groups over $\Z_{(p)}$, and $(\cP,i),(\cP_2,i_2)$ crystalline canonical models over each. Then we may canonically identify $\cP \times^{G^c} G_2^c \cong f^*\cP_2$ as weak filtered $F$-crystals with $G_2^c$-structure.
\end{enumerate}
\end{prop}
\begin{proof}
For the uniqueness, the key fact is that each component of $\cS_{G(\Zp)}$ has at least one crystalline point, which may be seen using special points via the argument of Kisin \cite[2.2.4]{kis2}. With this in hand, uniqueness of Frobenius, since it is a horizontal section, is immediate from the F condition (see for example \cite[1.18]{og}). 

It therefore suffices to check the uniqueness of lattices. Suppose we have two models $(\cP,i), (\cP',i')$. The question is whether $i^{-1} \circ i':\cP'_{E_v} \rightiso \cP_{E_v}$ extends over $\cO_v$ to give a canonical isomorphism of the models. By \cite[4.1.4]{tl} this is reduced to a question about whether the lattices induced by the fibre functors on each side are the same, and by (\ref{lattices on points}) and the L condition this follows, again given our observation that each component of $\cS_{G(\Zp)}$ contains a crystalline point.

For (2), we note that any crystalline point $x$ of $\cS_{G(\Zp)}$ can also be viewed via $\cS_{G(\Zp)} \rightarrow \cS_{G_2(\Zp)}$ as a crystalline point of $\cS_{G_2(\Zp)}$ (it is easy to see that the corresponding \'etale $G_2^c(\Zp)$-bundles over each are identified). This, together with the L condition and (\ref{lattices on points}), is enough to check the desired identification of lattices, and the Frobenius actions also agree because they will agree at a crystalline point by $p$-adic Hodge theory and the F condition and since they are horizontal sections of the isocrystal $\uHom(F^*\omega_{f^*\cP_2}[1/p], \omega_{f^*\cP_2}[1/p])$ this is enough.
\end{proof}

\subsection{Special type case}
\label{3.2}
Consider a pair $(T,h)$ where $T/\Q$ is a torus that splits over a number field that is unramified at $p$ and $h:\bS \rightarrow T_{\R}$ a map of real groups. This defines a Shimura datum that gives rise to a zero-dimensional Shimura variety with an obvious smooth integral model. In this short section we give a direct construction of the crystalline canonical model in this case.

\subsubsection{}
Fix (abusing notation) $T/\Z_{(p)}$ the integral model of $T$ obtained by the usual \'etale descent, let $E=E(T,h)$ and take $K = K^p T(\Zp) \subset T(\Af)$ a sufficiently small compact open so we obtain a Shimura variety 
$$\Sh_K(T,h) = \coprod_{i} \Spec E_i$$
with each $E_i/\Q$ unramified at $p$. The tower $\Sh_{K^p}(T,h) \rightarrow \Sh_K(T,h)$ is a $T^c(\Zp)$-Galois cover that gives rise to a fibre functor
$$\omega_{et}: \Rep_{\Zp}(T^c) \rightarrow \prod_{i} \Rep_{\Zp}(\Gamma_{E_i}).$$
For $v$ a place of $E$ over $p$, taking the rigid fibre of the $v$-adic completion corresponds to $\otimes_E E_v$ and localising the above functor we have (re-indexing the completions of every $E_i$ by $j$)
$$\omega_{et}: \Rep_{\Zp}(T^c) \rightarrow \prod_j \Rep_{\Zp}(\Gamma_{E_j}).$$

Let $\cS_{K} = \coprod_{j} \Spf \cO_{E_j}$ be the obvious smooth canonical integral model over $W := \hat{\cO_{E,v}}$.

\subsubsection{}
Let us first note that for any $V \in \Rep_{\Zp}(T^c)$ and any $j$, $\rho := \omega_{et}(V)_j[1/p]: \Gamma_{E_j} \rightarrow GL(V_{\Qp})$ is a crystalline representation. Indeed, this is well-known and follows by combining Deligne's definition of a canonical model of a special Shimura variety (e.g. \cite[2.2]{del2}) with a direct computation showing that unramified Hecke characters yield crystalline Galois representations (e.g. \cite[1.6]{far1}).

Therefore we may apply Fontaine's $D_{crys}$ functor to the image of $\omega_{et}[1/p]$ to obtain a filtered $F$-isocrystal with $G$-structure
$$\omega_{crys,\eta}: \Rep_{\Qp}(T^c) \rightarrow \FCrys_{\cS_{K}/W}[1/p].$$
The CPLF condition then forces us to take inside this the lattice defined by
$$\Rep_{\Zp}(T^c) \ni V \mapsto \omega_{crys}(V) := \phi^*\fM(\omega_{et}(V))/u \subset \omega_{crys,\eta}(V[1/p]).$$

To show that this gives a crystalline canonical model it remains to check the following.

\begin{prop}
\label{special strong}
Assume $p>2$. With notation as above, $\omega_{crys}(V) \subset \omega_{crys,\eta}(V[1/p])$ is flatly filtered and a strongly divisible lattice.
\end{prop} 
\begin{proof}
First note by functoriality that we may reduce to the case where $T=E^*$ and $h=h_E$ is the map defined as follows. Let $\tau_0:E \subset \C$ be the canonical embedding. If $\tau$ is real we let
$$\mu: \bS \rightarrow E^*_{\R} = \prod_{\tau} E_\tau^*$$
be defined by 
$$\mu(z) = (z\bar{z},1,\dots,1)$$
with the nonzero entry in place $\tau_0$. If $\tau_0$ is complex we let
$$\mu(z) = (z,1,\dots,1).$$
With this definition it is easy to check that we get a canonical map of Shimura data induced by $N_{E/\Q} \circ \mu$
$$(E^*,h_E) \rightarrow (T,h),$$
pulling back along which we obtain the reduction required.

Since the result is local and may be checked after passing to an unramified extension of the base, it will suffice to check that the Kisin modules obtained from the Galois representations of the form
$$\Gal(E_v^{ab}/E_{v}^{ur}) = \cO_{v}^* \map{\rho} GL(V)$$
give rise to flatly filtered strongly divisible lattices, where $(V,\rho)$ is an algebraic representation of $\bG_{m,\cO_{v}}$ 
and where we recall that $E_v/\Qp$ is unramified. Since the functors involved are $\otimes$-functors it suffices to check on a faithful representation.

Thus to finish, we note that $\cO_{v}^* \hookrightarrow GL_{\Zp}(\cO_v)$ is a faithful representation and that the Hodge-Tate weights of the associated Galois representation are $0$ and $1$, so since $p>2$ we may apply the theory of Fontaine-Laffaille and (\ref{FL equality}) to deduce the result.

\end{proof}

\subsection{Hodge type case}
\label{hodge case}
In this section we review Kisin's construction \cite[\S2]{kis2} of integral canonical models for Shimura varieties in the Hodge type case and show that it can be extended to give a construction of crystalline canonical models. The argument is essentially the same as that of \cite[\S4.5]{tl} but with less clutter and we give a full account for the reader's convenience.

\subsubsection{}
\label{3.3.1}
We recall the setup (and some of the results) of \cite[\S2.3]{kis2}. Let $(G,\fX)$ be a Hodge type Shimura datum with $G$ unramified at $p$ and fix a reductive integral model $G_{\Zp}/\Zp$ of $G_{\Qp}$. We may arrange a symplectic embedding $i: G \hookrightarrow GL(V)$ and a $\Z$-lattice $V_{\Z} \subset V$ such that the closure of the image of $i$ in $GL(V_{\Z_{(p)}})$ is a reductive subgroup $G_{\Z_{(p)}}$ whose base change to $\Zp$ agrees with $G_{\Zp}$. 

We now abuse notation and let $G$ denote this reductive $\Z_{(p)}$ group and $V := V_{\Z_{(p)}}$ or any base change thereof, provided the intended base is clear from context and no confusion will result. We also \cite[1.3.2]{kis2} fix a finite collection of tensors $s_\alpha \in V_{\Z_{(p)}}^\otimes$ such that that $G_{\Z_{(p)}} = GL_{s_\alpha}(V_{\Z_{(p)}})$, fix $v|p$ a place of $E=E(G,\fX)$, and (noting that $E$ is absolutely unramified at $p$ in such a context) let $W = W(k_v) = \cO_v$ with field of fractions $E_v$. We also remark that it is easy to see $G=G^c$ in this case.

Taking $K_p = G(\Zp)$ and $K'_p = \{g \in GSp_{\Qp}(V) | g(V_{\Zp}) = V_{\Zp}\}$ and $K^p \subset G(\A^{\infty,p})$ arbitrary open compact sufficiently small we may find $K' = K'^pK'_p$ such that $i$ induces a closed immersion of the Shimura variety with level $K = K_p K^p$ into a Siegel Shimura variety
$$\Sh_K(G,\fX) \hookrightarrow \Sh_{K'}(GSp,\bS^\pm).$$

The right hand side has a natural integral model $\cS_{K'}$ that is a moduli space of abelian varieties. Recall that the integral canonical model for the left hand side $\cS_{K}$ is defined as the normalisation of the closure 
$$\cS_K \map{\nu} \overline{\Sh_K} \subset \cS_{K'}$$
and is proved to be smooth over $W$.

Pulling back the universal abelian variety $\cA_{K'} \rightarrow \cS_{K'}$ we obtain an abelian variety $\cA_{K}/\cS_K$, and we may form its relative de Rham cohomology sheaf $$\cV = \cH^1_{dR}(\cA_K/\cS_K)$$
as a vector bundle on $\cS_K$ with Gauss-Manin connection and flat Hodge filtration satisfying Griffiths transversality. By \cite[2.3.9]{kis2} the absolute Hodge cycles $s_{\alpha,dR} \in \cV_{E_v}^\otimes$ extend to sections of $\cV^\otimes$. 

\subsubsection{}

We make some additional remarks in preparation for our construction. Combining Liu and Zhu's construction via \cite[3.9 (iv)]{lz} with a relative $p$-adic comparison result \cite[1.10]{sch}, and letting $\omega_{dR}$ be as in (\ref{3.1.3}) we see that there is a canonical identification $$\gamma: \cV_{E_v}^{an} \rightiso \omega_{dR}(V).$$ Moreover, $\gamma$ takes $s_{\alpha,dR}$ to $\omega_{dR}(s_{\alpha})$, as can be seen by looking at the generic point of each component and using the fact \cite{blasius} that absolute Hodge cycles in abelian motives are de Rham, i.e. exchanged under the $p$-adic comparison theorems.

Furthermore, since $\cS_K$ is a smooth $W$-scheme, after $p$-adically completing we may identify (as modules with connection) $$\hat{\cV} = \hat{\cH}^1_{dR}(\cA_K/\cS_K) \cong \cH^1_{crys}(\cA_K \otimes \kappa_v/W).$$
For any choices $(U_\beta,\sigma_\beta)$ of local Frobenius lift on $\cS_K$, we obtain a horizontal Frobenius structure on $\hat{\cV}$, and since $p>2$ the theory of Fontaine-Laffaille-Faltings implies that with the connection and filtration on the de Rham side, this is strongly divisible.

Our aim for the rest of the section will be to prove the following.

\begin{thm}
\label{main 3.3}
The functors (as $K^p$ varies)
$$\cP_K := \uIsom_{s_\alpha}(V, \hat{\cV})$$
together with connection, Frobenius and Rees structure induced from those on $\hat{\cV}$ define (for all choices $\{(U_\beta,\sigma_\beta)\}$ of local Frobenius lift) a $G(\Afp)$-equivariant filtered $F$-crystal with $G$-structure and, under the identification

$$\gamma_*: \cP_{K}[1/p] = \uIsom_{s_\alpha}(V_{E_v}, \cV_{E_v}^{an}) \rightiso P_{K_pK^p}$$  induced by $\gamma$ form the crystalline canonical model for $(G,\fX)$ at the place $v$.
\end{thm}

\subsubsection{}
Note that whenever we write something like $\uIsom_{s_\alpha}(V,\hat{\cV})$ we mean the sheaf of trivialisations $t: V_U \rightiso \hat{\cV}_U$ such that under $t^\otimes: V_U^\otimes \rightiso \hat{\cV}_U^\otimes$ each $s_\alpha$ is identified with the corresponding $s_{\alpha,dR}$ (or some other tensor labelled by $\alpha$ and in practice no confusion will arise). Also note that whenever such trivialisations exist locally, then $\uIsom_{s_\alpha}(V,\hat{\cV})$ becomes locally isomorphic to $\underline{\Aut}_{s_\alpha}(V) = G$. Since $G$ is smooth and affine, ``locally'' can be taken to mean fpqc, fppf or \'etale, the existence of \'etale local sections will be automatic by fpqc descent, and the presence of local sections is enough to conclude that $\uIsom_{s_\alpha}(V,\hat{\cV})$ is a $G$-bundle.

\subsubsection{}
We therefore first set out to show that $\cP_K$ admits a section in an fpqc neighbourhood of any (characteristic $p$) point $x \in \hat{\cS}_K$. In particular it will suffice to construct a section over the formal neighbourhood $\hat{N}_x = \Spf \hat{\cO}_{S_K,x}$ of $x$. To do this we need to recall more of the setup of Kisin's paper. We begin by quickly reviewing important aspects of Faltings' model for such a formal neighbourhood \cite[\S1.5]{kis2}.

\subsubsection{}
\label{3.3.6}
Suppose we are given a $p$-divisible group $\cG_0$ over a finite field $\kappa/k_v$. This has a (contravariant) Dieudonn\'e crystal $M_0=\D(\cG_0)(W(\kappa))$, with a Frobenius $\phi: \sigma^*M_0 \rightarrow M_0$ and an induced mod $p$ filtration $\overline{\Fil}^1 \subset M_0\otimes_{W(\kappa)} \kappa$. Since $k$ is a field, we may suppose this filtration is given by a cocharacter $\mu_0:\Gm \rightarrow GL(M_0\otimes \kappa)$. We can choose a lift $\mu:\Gm \rightarrow GL(M_0)$ and let $U$ denote the opposite unipotent to the parabolic defined by $\mu$. Completing along the identity section, we get an affine formal scheme $\Spf R$ that is noncanonically isomorphic to a power series ring $R \cong W(\kappa)[[t_1,...,t_n]]$. Making such a choice of co-ordinates, we can define a lift of Frobenius by $\sigma: t_i \mapsto t_i^p$. 

With this setup in place, we let $M=M_0 \otimes_{W(\kappa)} R$ with the (constant) filtration induced by $\mu$ and a semilinear Frobenius given by $$\phi: \sigma^*M \map{\phi \otimes \sigma} M \map{u} M,$$
where $u \in U(R) \subset GL(M)$ is the tautological $R$-point of $U$. 

\begin{thm} 
\label{faltings defo}
The module $M$ admits a unique integrable connection $\nabla: M \rightarrow M \otimes_R \Omega^1_{R/W(\kappa)}$ giving $M$ the structure of a filtered $F$-crystal. There is a $p$-divisible group $\cG_R$ over $R$, with a canonical identification of filtered $F$-crystals
$$\D(\cG_R)(R) \cong M.$$
Moreover, $\cG_R/\Spf R$ is a versal deformation of $\cG_0$.
\end{thm}
\begin{proof}
This is all done in Moonen \cite[4.5]{moon}, using a result of Faltings for the existence of the connection, and the theorems of de Jong and Grothendieck-Messing on essential surjectivity of Dieudonn\'e and ``admissible filtration'' functors for existence of the $p$-divisible group. 
\end{proof}

\subsubsection{}
Let $\cG$ be the lift of $\cG_0$ to $W(\kappa)$ corresponding to $\mu$ under Grothendieck-Messing (equivalently that obtained by pulling back $\cG_R$ along the zero section over $W(\kappa)$). Let $K=W(\kappa)[1/p]$ and suppose we are given an identification
$$T_p(\cG_{K})^* \cong V_{\Zp}$$
of the dual Tate module of (the generic fibre of) $\cG$ with the lattice $V_{\Zp}$. This equips this Tate module with tensors $s_\alpha$, and we assume further that these are $\Gamma_K$-invariant.

\begin{prop}
\label{crystensors}
The $p$-adic comparison isomorphism
$$M_0 \otimes_{W(\kappa)} B_{crys} \rightiso T_p(\cG_{K})^* \otimes_{\Zp} B_{crys}$$
identifies $s_\alpha$ with $\phi$-invariant tensors $s_{\alpha,0}$ in $Fil^0$ of $M_0^\otimes$.

Moreover, there exist (non-canonical) isomorphisms
$$V_{\Zp} \otimes_{\Zp} W(\kappa) \rightiso M_0$$
exchanging these tensors, and the mod $p$ filtration on $M_0 \otimes_{W(\kappa)} \kappa$ is induced by a cocharacter
$$\mu_0: \G_{m,\kappa} \rightarrow G_\kappa.$$
\end{prop}
\begin{proof}
This is an immediate application of \cite[1.3.6 (1)]{kis2} and \cite[1.4.3]{kis2}.
\end{proof}

\subsubsection{}
\label{3.3.10}
Given this, we now restrict our attention to those deformations of $\cG_0$ which ``respect the tensors''. Given we could choose the lift $\mu$ arbitrarily, we may assume it factors though a $G$-valued lift $\mu: {\Gm}_{,W(\kappa)} \rightarrow G_{W(\kappa)}$ by the smoothness result in [SGA III, XI 4.2]. Using this to produce an opposite unipotent $U_G$, and completing along the identity we obtain a formally smooth closed sub-formal scheme $z: \Spf R_G \hookrightarrow \Spf R$, on which we also fix a Frobenius lift $\sigma_{R_G}$ compatible with that on $R$. We can restrict $\cG_R$ to obtain $\cG_{R_G}$, and the identification of Dieudonn\'e modules (\ref{faltings defo}) gives rise to a canonical identification 
$$M_G := M_0 \otimes R_G = M \otimes_R R_G \rightiso \D(\cG_{R_G})(R_G).$$

\subsubsection{}
We now, following \cite[\S2.3]{kis2}, relate this discussion back to $\hat{N}_x = \Spf \hat{\cO}_{S_K,x}$. Let $\kappa = \kappa(x)$ and let $x' \in \cS_{K'}$ be the image of $x$, also viewed as a $\kappa$-point. We will study the formal neighbourhood of $x'$ in the closed subscheme $\overline{\Sh_K} \subset \cS_{K'}$ $$\hat{U}_{x'} = \Spf(\hat{\cO}_{\overline{\Sh_{K}},x'}).$$

Since $\overline{\Sh}_K \subset \cS'_{K'}$, we can use the universal abelian variety to associate to $x'$ a $p$-divisible group $\cG_0/\kappa$ and a deformation $\cG/\hat{U}_{x'}$ of this to the formal neighbourhood. Let $R$ be the versal deformation ring constructed above (\ref{3.3.6}), taking as input the $p$-divisible group $\cG_0$. We may fix a map $g:\hat{U}_{x'} \hookrightarrow \Spf R$ giving rise to $\cG$, injective by the Serre-Tate theorem and the fact that polarisations and prime to $p$ level structures lift uniquely through nilpotents.

Now suppose $\tilde{x} \in \Sh_K(F)$ (for $F/E$ a finite extension) is a point in characteristic zero specialising to $x$ inside $\cS_K$. Then using (\ref{crystensors}), the $G_F$-invariant tensors $s_{\alpha, et, \tilde{x}}$ on $H^1_{et}(\cA_{\tilde{x},\bar{K}}, \Zp)$ correspond to $\phi$-invariant $\Fil^0$ tensors $s_{\alpha, crys, \tilde{x}} \in \D(\cG_0)(W(\kappa))^\otimes$, which define a closed $W(\kappa)$-group subscheme $G(\tilde{x}) \subset GL(\D(\cG_0)(W(\kappa)))$ that is isomorphic to $G_{W(\kappa)}$ canonically up to inner automorphisms, and has the property that the filtration on $\D(\cG_0)(\kappa)$ is induced by a cocharacter $\mu_0: \Gm \rightarrow G(\tilde{x})_\kappa$. As described above (\ref{3.3.10}), this gives rise to a canonical closed formally smooth formal subscheme of the versal deformation space $\Spf R_{G(\tilde{x})} \subset \Spf R$.

Following the argument of \cite[2.3.5]{kis2}, we prove this in fact gives a model for the formal neighbourhood $\hat{N}_x$ of $x$ in $\cS_K$.
\begin{thm}
\label{formalid}
Let $\tilde{x}$ be as above, and let $Z \subset \hat{U}_{x'}$ be the irreducible component containing the image of $\tilde{x}$. Then $g|_Z$ factors through $\Spf R_{G(\tilde{x})}$ and in fact induces an isomorphism
$$ g|_Z: Z \rightiso \Spf R_{G(\tilde{x})}.$$
Moreover, this isomorphism extends to an isomorphism $\cG|_{Z} \rightiso \cG_{G(\tilde{x})}$ of $p$-divisible groups, and the induced isomorphism of Dieudonn\'e modules identifies the tensors on $\hat{\cV}|_{Z}^\otimes$ with the tensors on $\D(\cG_{G(\tilde{x})})(R_{G(\tilde{x})})^\otimes$.
\end{thm}
\begin{proof}
Fix an extension of $v$ to $\bar{E}$. We claim that for any finite extension $F \subset F'\subset \bar{E}$ and any point $\tilde{x}' \in \Sh_K(F')$ specializing to $x$ that also lies in $Z(F'_v)$, we have an identification of $s_{\alpha, crys, \tilde{x}'}= s_{\alpha, crys, \tilde{x}} \in \D(\cG_0)(W(\kappa))^\otimes$. By \cite[1.5.11]{kis2} this implies all the induced maps $\tilde{x}': \Spec F' \rightarrow \Spf R$ factor through the same adapted deformation space $\Spf R_{G(\tilde{x})} \subset \Spf R$.

This suffices to prove the result, as $\cO_Z$ being a quotient of a completion of a Jacobson ring and a domain, the intersection of all the prime ideals cut out by such points is zero. Assuming the claim, we therefore get the factorisation $g|_Z: Z \rightarrow \Spf R_{G(\tilde{x})}$. Moreover, since the target is a power series ring and $Z$ is irreducible of the same dimension, this map is an isomorphism. We have immediately an induced isomorphism $\cG|_Z \rightiso \cG_{G(\tilde{x})}$ since both these $p$-divisible groups are obtained as pullbacks of $\cG_R$ along the same map up to the isomorphism $g|_Z$. 

Finally, we must prove the tensors are identified under the induced
$$\cV|_{Z}^\otimes = \D(\cG|_Z)(\cO_{Z})^\otimes \rightiso \D(\cG_{G(\tilde{x})})(R_{G(\tilde{x})})^\otimes.$$
But recalling that $\D(\cG_{G(\tilde{x})})(R_{G(\tilde{x})}) \cong \D(\cG_0)(W(\kappa)) \otimes_{W(\kappa)} R_{G(\tilde{x})}$ with tensors coming from the first factor, and the Berthelot-Ogus comparison $H^1_{dR}(\cA_{\tilde{x'}}/F'_v) \cong \D(\cG_0)(W(\kappa)) \otimes_{W(\kappa)} F'_{v}$, together with the compatibility of this with $p$-adic comparison maps and the result of Blasius and Wintenberger \cite{blasius} stating that the tensors $s_{\alpha, et, \tilde{x}'}$ and $s_{\alpha, dR, \tilde{x}'}$ correspond under these maps, the tensors are identified (using the claim) on our dense set of points, and so since $\cV_{dR}$ is locally free they agree on the whole of $Z$.

It remains to prove the claim. But note that $s_{\alpha,dR}$ is parallel with its evaluation on fibre at $\tilde{x}'$ identified with $s_{\alpha, crys, \tilde{x}'}$, $Z_\eta$ is connected, and $\tilde{x},\tilde{x'}$ both specialise to $x$. By the theory of Berthelot-Ogus \cite[2.9]{bo1}, $\cV|_{Z}^\otimes[1/p]$ can be realised as a convergent $F$-isocrystal, so we can relate the fibres at $\tilde{x}$ and $\tilde{x'}$ by parallel transport and the claim is immediate.
\end{proof}
 
\begin{cor}
\label{formalcor}
Let $x \in \cS_K$ be a closed point of characteristic $p$, with residue field $\kappa$. Let $\hat{N}_x$ be its formal neighbourhood,  $\cG_0$ the $p$-divisible group over the special fibre and $\tilde{x}$ a point on the generic fibre specialising to $x$. Then the corresponding tensors $s_{\alpha, crys, \tilde{x}} \in \D(\cG_0)(W(\kappa))^\otimes$ are independent of the choice of $\tilde{x}$, so $G(x):=G(\tilde{x})$ is well-defined up to inner automorphisms, and we have a canonical identification $\hat{N}_x \cong \Spf R_G$ extending to an identification of filtered $F$-crystals exchanging the tensors.
\end{cor}

\subsubsection{}
We may use this to finally prove the main theorem (\ref{main 3.3}). Recall that to show $\cP_K$ was a $G$-bundle it would suffice to construct a section over the formal neighbourhood $\hat{N}_x = \Spf \hat{\cO}_{S_K,x}$ of $x$. But combining a map as in the second part of (\ref{crystensors}) with the identifications (\ref{3.3.10}) and (\ref{formalcor}) we get a tensor preserving isomorphism
$$(V \otimes_{\Zp} W(\kappa)) \otimes_{W(\kappa)} R_{G(\tilde{x})} \rightiso M_0 \otimes_{W(\kappa)} R_{G(\tilde{x})} \rightiso \D(\cG_{R_{G(\tilde{x})}}) \rightiso \cV|_{\hat{N}_x}$$ 
that is exactly such a section.

\subsubsection{}
We next need to show that the torsor $\cP_K = \uIsom_{s_\alpha}(V,\hat{\cV})$ is a filtered $F$-crystal with $G$-structure. Since the $s_{\alpha,dR}$ are horizontal by construction, $\cP_K$ acquires a flat connection by (\ref{2.3.2}). 

We define its filtration via the Rees construction. Since $s_{\alpha,dR} \in \Fil^0(\hat{\cV}^\otimes)$, we have by (\ref{rees equivalence}) $\Rees(s_{\alpha,dR}) \in \Rees(\hat{\cV})^\otimes$ a finite set of $\Gm$-invariant tensors. Let $\cP_K$ have Rees structure given by
$$\Rees(\cP_K) := \uIsom_{\Rees(s_\alpha)}(V_{\hat{S}_K \times \A^1}, \Rees(\hat{\cV}))$$
together with the natural identification 
$\Rees(\cP_K)/(t-1) = \cP_K$. Taking local splittings of the Hodge filtration it is easy to see that $\Rees(\cP_K)$ is a $\Gm$-equivariant torsor, and this suffices for it to define a Rees structure. Moreover, we have Griffiths transversality for this Rees structure and the connection on $\cP_K$: it is inherited from Griffiths transversality for $\hat{\cV}$ via (\ref{2.3.4} (1)).

Finally we must define a Frobenius using the Frobenius $\phi: \hat{\cV}[1/p] \rightiso \hat{\cV}[1/p]$ coming from crystalline cohomology, and check it satisfies all necessary conditions. First note that the $s_{\alpha,dR} \in \hat{\cV}^\otimes$ are $\phi$-invariant. Indeed, since they and the Frobenius $\phi$ coming from relative crystalline cohomology are horizontal, it suffices to check this on a single point of each component, whence it follows from the first part of (\ref{crystensors}). Given this, $\phi$ induces the structure of a filtered $F$-isocrystal with $G$-structure via
$$\uIsom_{s_\alpha}(V[1/p],\hat{\cV}[1/p]) \ni s \mapsto \phi \circ s \in \uIsom_{s_\alpha}(V[1/p],\hat{\cV}[1/p]).$$

What is more, for any choices $(U_\beta,\sigma_\beta)$ of local Frobenius lifts on $\hat{\cS}_K$, it is well-known that (as the relative crystalline cohomology of an abelian variety) $\hat{\cV}$ is a strongly divisible filtered $F$-crystal, so $\phi$ induces isomorphisms between lattices inside $\hat{\cV}[1/p]$ given by
$$\Rees(\sigma_{\beta}^*\hat{\cV})/(t-p) \rightiso \hat{\cV}|_{U_\beta}$$
which respects the tensors, and therefore induces an isomorphism
$$\Rees(\sigma_\beta^*\cP_K)/(t-p) = \uIsom_{s_\alpha}(V,\Rees(\sigma_{\beta}^*\hat{\cV})/(t-p)) \rightiso \uIsom_{s_\alpha}(V,\hat{\cV}|_{U_\beta}) = \cP_K.$$
Thus $\cP_K$ with the structures we have described is a filtered $F$-crystal with $G$-structure.

\subsubsection{}
To complete the proof of the main theorem of this section (\ref{main 3.3}) it suffices to check that the lattice thus constructed is the crystalline canonical model, of which there remains only the L condition. Take $x \in \cS_{K}(W(\kappa))$ a crystalline point. Recall that $V$ is our faithful representation of $G=G^c$, here taken over $\Zp$. The L condition will follow if we can show that inside $H^1_{crys}(\cA_x/W(\kappa))[1/p] = D_{crys}(\hat{V}_p(\cA_x)^*)$ we have an identity of lattices $$x^*\cP_K \times^G V = \phi^*\fM(x^*\omega_{et}(V))/u.$$

The left hand side is easily computed to be simply $\D(\cA_x)(W(\kappa))$, which in the case $p>2$, noting that here the Hodge Tate weights are 0 and 1, is the Fontaine-Laffaille module attached to $\hat{T}_p(\cA_x)^*$, and this is equal to the right hand side by (\ref{FL equality}). For $p=2$, one must give a different argument, for example deducing the equality of lattices we need from the result of Kim \cite[4.2 (1)]{kim}. With this, our main theorem in the Hodge type case is proved.

\subsection{Abelian type case}
\label{abtype}
In this section we extend the results of the previous section to the case of abelian type Shimura varieties. We follow the same basic approach as in \cite[\S4.7]{tl}, first using the Hodge type case construction to obtain a construction over an ``initial Shimura datum'' from which we may push it forward to the abelian type Shimura datum.

One slight variation is that working at hyperspecial level almost everywhere in \cite{tl}, we were forced to work over a connected component and then deduce the existence over the whole Shimura variety by a formal argument. In the present context, with hyperspecial level at a single prime we can borrow the group theoretic description of \cite[3.3.10]{kis2} to give a direct construction more analogous to that of Deligne \cite{del2}, and which we suspect might be friendlier for computations.

\subsubsection{}
We recall the basic setup. Fix $(G_2,\fX_2)$ a Shimura variety of abelian type with $G_2$ unramified at $p$ and fix a reductive integral model we will also denote by $G_2/\Zbp$ (or $G_{2,\Zbp}$ if there is risk of confusion). By the argument of \cite[3.4.14]{kis2} we may find a Hodge type Shimura datum $(G,\fX)$ and a reductive integral model $G_{\Zbp}/\Zbp$ for $G$ such that there is a central isogeny $i:G^{der} \rightarrow G_2^{der}$ which induces a map of connected Shimura data $(G^{der},\fX^+) \rightarrow (G_2^{der},\fX_2^+)$ and under which $G^{der}(\Zbp)$ is mapped to $G_2^{der}(\Zbp)$. We may also arrange that $E:=E(G,\fX) \supset E(G_2,\fX_2) =: E_2$.

Next recall the following result of \cite[\S4.6]{tl}.
\begin{prop}[Initial Shimura data]
Let $(G^{der},\fX^+)$ be a connected Shimura datum, and $E \supset E(G^{ad},\fX^+_{ad})$ a field such that there exist Shimura data with connected datum $(G^{der},\fX^+)$ and reflex field contained in $E$.

Then there is an ``initial object'' $(\cB,\fX_{\cB})$ amongst such Shimura data, in the sense that for any such Shimura datum $(G,\fX)$ there is a canonical map $(\cB,\fX_{\cB}) \rightarrow (G,\fX)$ of Shimura data. The group $\cB$ is an extension
$$1 \rightarrow G^{der} \rightarrow \cB \rightarrow E^* \rightarrow 1,$$
the canonical map given by the projection $$\cB = G \times_{G^{ab}} E^* \rightarrow G,$$
where the map $E^* \rightarrow G^{ab}$ is $N_{E/\Q} \circ \mu$,
for the abelianised Hodge cocharacter $\mu: E^* \rightarrow G^{ab}_E$ 
and the construction enjoys the following two functoriality properties.

\begin{enumerate}
\item Let $j: (G^{der},\fX^+) \rightarrow (G_2^{der},\fX^+_2)$ be a map of connected Shimura data (induced by a central isogeny $G^{der} \rightarrow G_2^{der}$), and $(\cB,\fX_{\cB})$, $(\cB_2,\fX_{\cB,2})$ the corresponding initial Shimura data. Then $j$ induces a central isogeny $j_*:\cB \rightarrow \cB_2$ that gives a map of Shimura data
$$j_*: (\cB,\fX_\cB) \rightarrow (\cB_2,\fX_{\cB,2})$$

\item Fix $(G^{der},\fX^+)$, let $E \supset E' \supset E(G^{ad},\fX^+_{ad})$ be two fields as above giving initial Shimura data $(\cB_E,\fX_{E}),(\cB_{E'}, \fX_{E'})$. Then the \emph{norm map} $N_{E/E'}:E^* \rightarrow E'^*$ induces a canonical morphism of Shimura data
$$N_{E/E'}: (\cB_E,\fX_{E}) \rightarrow (\cB_{E'}, \fX_{E'}).$$
\end{enumerate}
\end{prop}

\subsubsection{}
Let $E=E(G,\fX)$ and $E_2 = E(G_2,\fX_2)$. As in \cite[4.7.3]{tl} we note that since both are unramified at $p$, we have natural integral models $E^*_{\Zbp},E_{2,\Zbp}^*$ and these give rise to a diagram
$$G_{\Zbp} \leftarrow G_{\Zbp} \times_{G^{ab}_{\Zp}} E^*_{\Zbp} = \cB_{\Zbp} \rightarrow \cB_{2,\Zbp} \rightarrow G_{2,\Zbp}.$$
This in turn gives a diagram of integral models, over which our entire construction will take place
$$\cS_{G(\Zp)}(G,\fX) \lmap{a} \cS_{\cB(\Zp)}(\cB,\fX_\cB) \map{b} \cS_{G_2(\Zp)}(G_2,\fX_2).$$
Let us also recall the map $$\delta: \cS_{\cB(\Zp)}(\cB,\fX_\cB) \rightarrow \cS_{\hat{\cO_E}^*}(E^*,h_E)$$
coming from projection onto the second factor of $\cB = G \times_{G_{ab}} E^*$. 

\subsubsection{}
\label{3.4.4}
These remarks made, we may take the crystalline canonical models $\cP_{G(\Zp)}$ of (\ref{main 3.3}), and $\cP_{\hat{\cO_E}^*}$ of (\ref{3.2}), and by the same argument as \cite[4.7.5]{tl} manufacture a candidate (viewed as a weak filtered $F$-crystal)
$$\cP_{\cB(\Zp)} = a^*\cP_{G(\Zp)} \times_{\theta} \delta^*\cP_{\hat{\cO_E}^*}$$
for the crystalline canonical model over $\cS_{\cB(\Zp)}$, where $\theta$ is a natural isomorphism between the pushforwards of these weak filtered $F$-crystals with $G$ structure to the group $G^{ab}$, and the Rees structure on a fibre product is given by the formula
$$\Rees(P_1 \times_\theta P_2) := \Rees(P_1) \times_{\Rees(\theta)} \Rees(P_2)$$
which relies on the fact, easily seen in this case, that $\theta$ induces an isomorphism of filtered $G^{ab}$-bundles.

Let us check this is indeed the crystalline canonical model. Firstly, we must canonically identify $\cP_{\cB(\Zp)} \otimes E$ with $P_{\cB(\Zp)}(\cB,\fX_{\cB})$. Recall (\ref{lz facts} (1)) that we may canonically identify
$$P_{\cB(\Zp)}(\cB,\fX_{\cB}) \times^{\cB^c} G \cong a^*P_{G(\Zp)}(G,\fX)$$
and $$P_{\cB(\Zp)}(\cB,\fX_{\cB}) \times^{\cB^c} E^{*,c} \cong \delta^*P_{\hat{\cO_E}^*}(E^*,h_E)$$
 and so by \cite[4.7.4]{tl} the identification required follows immediately from the $\theta$-compatible identifications 
 $$a^*(\iota_G): a^*\cP_{G(\Zp)} \otimes E \rightiso a^*P_{G(\Zp)}(G,\fX)$$
 and $$\delta^*(\iota_{E^*}): \delta^*\cP_{\hat{\cO_E}^*} \otimes E \rightiso \delta^*P_{\hat{\cO_E}^*}(E^*,h_E).$$
 
 Next, we note that since there is a commutative diagram of pro-\'etale torsors
 $$
 \begin{CD}
     \Sh(\cB,\fX_{\cB}) @>>> \Sh(G,\fX)  \\
@VVV        @VVV\\
  \Sh_{\cB(\Zp)}(\cB,\fX_{\cB})   @>>>  \Sh_{G(\Zp)}(G,\fX),
\end{CD}
 $$
and a similar one replacing $(G,\fX)$ by $(E^*,h_E)$, we have that any crystalline point $x \in \cS_{\cB(\Zp)}$ maps to a crystalline point of $\cS_{G(\Zp)}(G,\fX)$ and $\cS_{\hat{\cO_E}^*}(E^*,h_E)$, and that $\cP_{\cB(\Zp)} \times^{\cB^c} G$ and $\cP_{\cB(\Zp)} \times^{\cB^c} E^{*,c}$ satisfy the CPLF condition. Applying the uniqueness statement of \cite[4.7.4]{tl} over each crystalline point $x$, this is enough to deduce that $\cP_{\cB(\Zp)}$ satisfies the CPLF condition.

Next we must check that the equivariant $\cB(\Afp)$ action on $P_{\cB(\Zp)}(\cB,\fX_\cB)$ extends to $\cP_{\cB(\Zp)}$. But this follows formally from the uniqueness statement (\ref{main uniqueness}) whose proof one can easily check does not use $\cB(\Afp)$-equivariance or strong divisibility. 

At this point we know $\cP_{\cB(\Zp)}$ is a weak crystalline canonical model, and it remains to check strong divisibility at finite level $K$ for any local Frobenius lifts $(U_\beta, \sigma_\beta)$ on $\cS_{\cB(\Zp)K}$. Take an open compact $K' \subset G(\Afp)$ such that the identity Hecke operator induces a map of Shimura varieties $\cS_{\cB(\Zp)K} \rightarrow \cS_{G(\Zp)K'}$, and by \cite[2.0.13]{del2} (and the extension property) it is possible to arrange that this be an isomorphism when restricted to a connected component. Note that strong divisibility may be checked after base change to $\cO = \cO_v^{ur}$, over which both varieties split into a union of connected components and note that any component $\cS^+ \subset \cS_{\cB(\Zp)K,\cO}$ is identified with a component $\cS'^+ \subset \cS_{G(\Zp)K',\cO}$. Therefore any local Frobenius lift $(U_\beta,\sigma_\beta)$ on $\cS^+$ induces one on $\cS'^+$ relative to which the strong divisibility of $\cP_{\cB(\Zp)}$ follows from its identification with a crystalline canonical model for $(G,\fX)$.

\subsubsection{}
We now turn our attention to descending $\cP_{\cB(\Zp)}$ to a crystalline canonical model over $\cS_{G_2(\Zp)}(G_2,\fX_2)$. Let us first recall some of the group theoretic constructions of Kisin \cite[\S3.3]{kis2} following Deligne, and introduce some new ones of our own. 

The reader unfamiliar with $*$ notation can check these references for the precise definitions, but it usually suffices just to imagine $H *_\Gamma \Delta$ as arising in the following vague way. Suppose $H$ and $\Delta$ act on some object but these actions interact in some way: there is a subgroup $\Gamma$ of $H$ whose action is identified with that of a part of $\Delta$, and the actions of $H$ and $\Delta$ need not commute but interact via some conjugation action of $\Delta$ on $H$. Then $H *_\Gamma \Delta$ is the subgroup of the automorphism group generated by $H$ and $\Delta$ in such a situation.

We will also (as in \cite[3.3.2]{kis2}) write $G(\Zbp)_+ := G(\Zbp) \cap G(\Q)_+$, where $G(\Q)_+$ is the intersection of the preimage of the connected component $G^{ad}(\R)^+$ of $G^{ad}(\R)$ with $G(\Q)$, write $G^{ad}(\Zbp)^+ := G^{ad}(\Zbp) \cap G^{ad}(\R)^+$, and use overline notation to denote closure in $G(\Afp)$.

For $G$ a reductive group over $\Z_{(p)}$ with centre $Z$, we define the locally profinite group
$$\cA_p(G) := \frac{G(\Afp)}{\overline{Z(\Zbp)}} *_{\frac{G(\Zbp)_+}{Z(\Zbp)}} G^{ad}(\Zbp)^+,$$
and her subgroup
$$\cA^\circ_p(G) := \frac{\overline{G(\Zbp)_+}}{\overline{Z(\Zbp)}} *_{\frac{G(\Zbp)_+}{Z(\Zbp)}} G^{ad}(\Zbp)^+,$$
which \cite[3.3.2]{kis2} depends only on $G^{der}/\Zbp$, and so we may also denote by $\cA^\circ_p(G^{der})$.

We shall also need the extensions

$$\cA_p^{dR}(G) := (G^c \times \frac{G(\Afp)}{\overline{Z(\Zbp)}}) *_{\frac{G(\Zbp)_+}{Z_{nc}(\Zbp)}} G^{ad}(\Zbp)^+$$
and
$$\cA_p^{dR,\circ}(G) := (G^c \times \frac{\overline{G(\Zbp)_+}}{\overline{Z(\Zbp)}}) *_{\frac{G(\Zbp)_+}{Z_{nc}(\Zbp)}} G^{ad}(\Zbp)^+$$
where $\frac{G(\Zbp)_+}{Z_{nc}(\Zbp)}$ is the obvious subgroup embedded diagonally in $(G^c \times \frac{G(\Afp)}{\overline{Z(\Zbp)}})$, and $G^{ad}(\Zbp)^+$ acts by conjugation on both factors simultaneously. We view this merely as a functor from $\Zbp$-algebras to groups, although it is possible to make more refined statements.

Given our $p$-adic Hodge theory definition of $P_{G(\Zp)}(G,\fX)$ we need to check the following statement.

\begin{prop}
The group $\cA_{p}^{dR}(G)$ is canonically an extension
$$1 \rightarrow G^c \rightarrow \cA_{p}^{dR}(G) \rightarrow \cA_p(G) \rightarrow 1.$$
 Let $(G_\Q,\fX)$ be a Shimura datum. Then $\cA_p^{dR}(G)$ acts canonically on $P_{G(\Zp)}(G_\Q,\fX)$ in such a way that the normal subgroup $G^c$ acts via the $G^c$-bundle action, and the action is equivariant via the above projection for the $\cA_p(G)$-action on $\Sh_{G(\Zp)}(G_\Q,\fX)$.
 
 Moreover we have:
 \begin{enumerate} 
 \item For any weak filtered $F$-crystal with $G$-structure $\cP_{G(\Zp)}$ together with an identification $\iota: \cP_{G(\Zp),E(G_\Q,\fX)_v} \rightiso P_{G(\Zp)}(G_\Q,\fX)$, the CPLF condition is invariant under the $\cA_p^{dR}(G)$-action. 
 \item The group $\cA_p^{dR,\circ}(G)$ is precisely the subgroup of $\cA_p^{dR}(G)$ acting on the fiber over a connected component $\Sh_{G(\Zp)}(G_\Q,\fX)^+_{\bar{E_v}}$.
 \end{enumerate}
 \end{prop}
 \begin{proof}
 The first statement about $\cA_p^{dR}(G)$ expressed as an extension is immediate from basic properties of the * construction \cite[2.0]{del2}.
 
 Recall that by definition (\ref{3.1.3}) and the Tannakian formalism
 $$P_{G(\Zp)}(G_\Q,\fX) = \uIsom^\otimes(\omega_{dR,triv}, \omega_{dR})$$
 where 
 $$\omega_{dR, triv}: \Rep(G_{E_v}) \ni V \mapsto V \otimes_{E_v} \cO_{\Sh_{G(\Zp)}(G_\Q,\fX)} \in \Bun(\Sh_{G(\Zp)}(G_\Q,\fX)).$$
By construction and (\ref{lz facts} (2)) we have an action of $G^c \times \frac{G(\Afp)}{\overline{Z(\Zbp)}}$ on $P_{G(\Zp)}(G_\Q,\fX)$. Our task is therefore to define an additional $G^{ad}(\Zbp)^+$ action and check the following relations.
\begin{itemize}
\item If $\gamma \in G(\Zbp)_+$, then its image in $G^{ad}(\Zbp)^+$ and its diagonal image in $G^c \times \frac{G(\Afp)}{\overline{Z(\Zbp)}}$ act identically on $P_{G(\Zp)}(G_\Q,\fX)$.
\item If $\delta \in G^{ad}(\Zbp)^+$ and $h \in G^c \times \frac{G(\Afp)}{\overline{Z(\Zbp)}}$ then in the automorphism group of $P_{G(\Zp)}(G_\Q,\fX)$ we have the relation $\delta(h).\delta = \delta .h$ where $\delta(h)$ is the image of $h$ under conjugation by $\delta$.
\end{itemize}

First the definition of the action. Recall that $\omega_{dR} = \D^0_{dR} \circ \omega_{et}$ where $\D^0_{dR}$ is Liu and Zhu's functor from de Rham lisse sheaves to vector bundles, and in particular is a functor, and of course we also have $\omega_{dR,triv} = \D^0_{dR} \circ \omega_{\Qp,triv}$ in the obvious sense. 

Therefore given an automorphism $\alpha$ of $\Sh(G_\Q,\fX) = \uIsom^\otimes(\omega_{\Zp,triv}, \omega_{et,\Zp})$, we may (applying $\D^0_{dR}$ and canonically extending using the $G^c$-action, perhaps after etale localising and then descending using canonicity) obtain an automorphism $\D^0_{dR}(\alpha)$ of
$$P_{G(\Zp)}(G_\Q,\fX) = \uIsom^\otimes(\D^0_{dR} \circ \omega_{et,triv}, \D^0_{dR} \circ \omega_{et}).$$

But by \cite[2.1.13]{del2} there is an algebraic $G^{ad}(\Q)^+$-action on $\Sh(G_\Q,\fX)$ given explicitly over $\C$, in particular an explicit action of $G^{ad}(\Zbp)^+$. We therefore define our action of $G^{ad}(\Zbp)^+$ using this action and the above remarks. One may verify that the relations above, and the claim about equivariance for the $\cA_p(G)$ action on $\Sh_{G(\Zp)}(G_\Q,\fX)$ required follow immediately from Deligne's description. 

It remains to check (1) and (2). For (1), we need only remark that the action above is defined via the \'etale $G$-bundle $\omega_{et}$, and so in particular since the $G^{ad}(\Zbp)^+$-action leaves stable the $\Zp$-lattice $\omega_{et} \subset \omega_{et} \otimes \Qp$ and induces an automorphism of lisse sheaves, after applying $\D^0_{dR}$ it will exchange the lattices present in condition L and commute with the Frobenii present in condition F. Finally (2) follows from the main statement together with Kisin's proof \cite[3.3.7]{kis2} that $\cA_p^\circ(G) \subset \cA_p(G)$ is the stabiliser of $\Sh_{G(\Zp)}(G,\fX)^+_{\bar{E}_v} \subset \Sh_{G(\Zp)}(G,\fX)_{\bar{E}_v}.$
 
\end{proof}

\subsubsection{}
Next we remark that this action extends to a crystalline canonical model $\cP_{G(\Zp)}$. Indeed the argument follows formally from uniqueness (\ref{main uniqueness} (1)), together with the fact that the $\cA_p^{dR}(G)$-action preserves the CPLF condition.

In particular, let us return to our construction of an crystalline canonical model for $P_{G_2(\Zp)}(G_2,\fX_2)$. We have so far constructed a crystalline canonical model $\cP_{\cB(\Zp)}$ for $P_{\cB(\Zp)}(\cB,\fX_{\cB})$, and by the above argument this carries an action of the group $\cA_p^{dR}(\cB)$.

We now note some further constructions closely following \cite{kis2}, leaving the details (which are very similar to those of \cite{kis2}) to the reader. We re-emphasise that our entire picture takes place over the $p$-adic field $E_v$ where $E$ is the reflex field of $(G,\fX)$ and $v$ a fixed place, and have fixed $\bar{E}_v$ an algebraic closure. Let $L/E_v$ denote the maximal unramified extension of $E_v$.

We define $\cE^{dR}(\cB) \subset \cA_p^{dR}(\cB) \times \Gal(L/E_v)$ to be the subgroup that induces an automorphism of the fiber $P_{\cB(\Zp)}(\cB,\fX_\cB)^+ \subset P_{\cB(\Zp)}(\cB,\fX_\cB)_L$ over a connected component $\Sh_{\cB(\Zp)}(\cB,\fX_\cB)^+$ of the Shimura variety. As in \cite[3.3.7]{kis2} this enjoys sitting within a canonical identification
$$\cA_p^{dR}(\cB) *_{\cA_p^{dR,\circ}(\cB)} \cE^{dR}(\cB) \rightiso \cA_p^{dR}(\cB) \times \Gal(L/E_v).$$

Following \cite[3.3.10]{kis2} we may push this identification out along $\cA_p^{dR}(\cB) \rightarrow \cA_p^{dR}(G_2)$ to identify the groups 
$$\cA_p^{dR}(G_2) *_{\cA_p^{dR,\circ}(\cB)} \cE^{dR}(\cB) \rightiso \cA_p^{dR}(G_2) \times \Gal(L/E_v),$$
which act naturally on
$$(\cA_p^{dR}(G_2) \times P_{\cB(\Zp)}(\cB,\fX_\cB)^+_L)/ \cA_p^{dR,\circ}(\cB) \rightiso P_{G_2(\Zp)}(G_2,\fX_2)_{L}.$$

Next, note that by what we have already seen and that $L/E_v$ is unramified (with ring of integers which we denote by $\cO$), the $\cA_p^{dR}(G_2) *_{\cA_p^{dR,\circ}(\cB)} \cE^{dR}(\cB)$-action extends to an integral model and can be used to define a candidate crystalline canonical model
$$\cP_{G_2(\Zp),\cO} := (\cA_p^{dR}(G_2) \times \cP_{\cB(\Zp),\cO}^+)/ \cA_p^{dR,\circ}(\cB).$$
We should pause to note that the quotient makes sense because $\Ker(\cA_p^{dR,\circ}(\cB) \rightarrow \cA_p^{dR,\circ}(G_2))$ acts freely on $\cP_{\cB(\Zp),\cO}^+$, being an extension of the group scheme $\Ker(\cB^c \rightarrow G_2^c)$ which obviously acts freely on each fibre and the group $\Delta(\cB,G_2) := \Ker(\cA_p(\cB) \rightarrow \cA_p(G_2))$ which acts freely on the base by \cite[3.4.6]{kis2}. Finally, note that the $\Gal(L/E_v) = \Gal(\cO/\cO_v)$-action provides a descent datum giving us $\cP_{G_2(\Zp),\cO_v}$.

\subsubsection{}
Let us verify that the construction above does indeed define a crystalline canonical model for $P_{G_2(\Zp)}(G_2,\fX_2)_{E_v}$. First, note that the construction gives a weak filtered $F$-crystal with $G_2^c$ structure because we may work at finite level, where this is implied by descent along a finite \'etale map for filtered $F$-isocrystals and vector bundles. It also has an equivariant $G_2(\Afp)$-action built in by construction (via $G_2(\Afp) \subset \cA_p^{dR}(G_2)$).

The CPLF condition may now be checked easily using the $\cB(\Afp)$-equivariant diagram of Galois covers
$$
\begin{CD}
    \Sh(\cB,\fX_{\cB}) @>>>  \Sh_{\cB(\Zp)}(\cB,\fX_\cB)\\
@VVV        @VVV\\
  \Sh(G_2,\fX_2)   @>>>  \Sh_{G_2(\Zp)}(G_2,\fX_2)
\end{CD}
$$

and by noting that (using the transitivity of the $G_2(\Afp)$-action on components of $\Sh_{G_2(\Zp)}(G_2,\fX_2)$ \cite[3.3.10]{kis2}) one can lift a Hecke translate of any crystalline point $x \in \cS_{G_2(\Zp)}$ to a crystalline point $\tilde{x} \in \cS_{\cB(\Zp)}$.

Finally, for the strong divisibility condition, again we may use the Hecke action to find for any component $\cS_2^+ \subset \cS_{G_2(\Zp),\cO}$ a finite \'etale map $\cS^+ \rightarrow \cS_2^+$ where $\cS^+ \subset \cS_{\cB(\Zp),\cO}$ is a component and over which there is a natural finite \'etale map $\cP^+ \rightarrow \cP_2^+$ induced by restricting $\cP_{\cB(\Zp)}$ and $\cP_{G_2(\Zp)}$ to these respective components. For any local Frobenius lifts $(U_\beta,\sigma_\beta)$ on $\cS_2^+$, we may lift them uniquely (using \'etaleness) to local Frobenius lifts on $\cS^+$, and it is clear that for $\cP_2^+$ to be strongly divisible it suffices that $\cP^+$ be strongly divisible. But this is known because $\cP^+$ is a component of the crystalline canonical model for $(\cB,\fX_{\cB})$.

The last step is to check that in the case where $E_{2,v}:= E(G_2,\fX_2)_v \not= E_v$ that our construction descends to $\cO_{E_2,v}$. This almost follows immediately from the uniqueness statement (one obtains a descent datum from $P_{G_2(\Zp)}(G_2,\fX_2)_{E_{2,v}}$), with the only other statement necessary to prove being that the Frobenius we have above defined on $P_{G_2(\Zp)}(G_2,\fX_2)_{E_v}$ descends to $E_{2,v}$. This final part we postpone for (\ref{Frobenius descent remark}). Otherwise our main theorem has now been proved.

\subsection{Comparison with \'etale $G$-bundles}

With our construction (essentially) complete, the following is a natural question to ask. Let $(G,\fX)$ be a Shimura datum with $G$ unramified at $p$ and fix $G/\Zbp$ a reductive model, take $K^p \subset G(\Afp)$ open compact, let $E \supset E(G,\fX)$ absolutely unramified at $p$ and fix $v|p$ a place of $E$.

Using the pro-Galois $G^c(\Zp)$ cover $\Sh_{K^p} \rightarrow \Sh_{G(\Zp)K^p}$ we may define $$\omega_{et,K^p}:\Rep_{\Zp}(G) \rightarrow \Lisse_{\Zp}(\Sh_{G(\Zp)K^p}(G,\fX)_{E_v}).$$

On the other hand, the crystalline canonical model, if it exists, defines a fibre functor
$$\omega_{crys,K^p}: \Rep_{\Zp}(G) \rightarrow \FCrys_{\cS_{G(\Zp)K^p/\cO_v}}.$$

It is naturally to ask if these are integrally associated (recall from \ref{integrally associated}), and we would expect this to always be the case. Here we prove the following.

\begin{thm}
\label{assoc thm}
With the above setup, if $(G,\fX)$ is of abelian type then $\omega_{et,K^p}$ and $\omega_{crys,K^p}$ are integrally associated. 
\end{thm}
\begin{proof}
The proof is basically to follow the proof of our main theorem, which constructs $\omega_{crys,K^p}$, and check the required statement at each stage of the construction. Recall that being integrally associated has two parts, an `associated' statement (1), and a statement about Fontaine-Laffaille modules (2). We will also suppress the $K^p$ subscript where there is no danger of confusion.

In the special type case, part (1) follows directly from $p$-adic Hodge theory (we used Fontaine's construction to get $\omega_{crys}[1/p]$) and part (2) from the compatibility (\ref{FL equality}) of our Kisin module theoretic construction with the theory of Fontaine-Laffaille. In fact, for every case below once we have part (1), part (2) will follow from (\ref{FL equality}) and the CPL condition of crystalline canonical models, since equality of lattices may be checked at a point on each component, so it remains to check part (1) in the Hodge and abelian type cases.

In the Hodge type case, we apply the criterion (\ref{2.6.5}). We need only check the Hodge embedding $G \hookrightarrow GL(V)$ satisfies the conditions of this proposition. First, we need to remark that (letting $\cA \rightarrow \Sh_{G(\Zp)K^p}$ be the pullback of the universal abelian scheme corresponding to our Hodge embedding)
$$\bL(\omega_{crys}(V)) = \bL(\D(\cA)) \cong \hat{T}_p(\cA) = \omega_{et}(V),$$
where the middle $p$-adic comparison step is a well-known special case of (\ref{comp Zp}). We may also take the $s_\alpha$ as in (\ref{3.3.1}) and note that we already know (by applying the de Rham cycles result of \cite{blasius} to the generic fibre of each component) that $s_{\alpha,et}$ and $s_{\alpha,dR}$ are exchanged under the $p$-adic comparison isomorphism. Thus (\ref{2.6.5}) suffices for the Hodge type case.

For the abelian type case, we proceed in two steps. Using the notation of the previous section, we now know the statement for $(G,\fX)$ and $(E^*,h_E)$ and use this to prove it first for $(\cB,\fX_{\cB})$ and then $(G_2,\fX_2)$. Recall (\ref{3.4.4}) that the crystalline canonical model for $(\cB,\fX_\cB)$ is constructed as
$$\cP_{\cB(\Zp)} = a^*\cP_{G(\Zp)} \times_{\theta} \delta^*\cP_{\hat{\cO_E}^*}.$$
We argue by constructing a similar picture on the `Galois' side. Let us introduce the general notation (so similarly for $\cB,E^*$) $$\bP_{G(\Zp)} := \Sh(G,\fX) \rightarrow \Sh_{G(\Zp)}(G,\fX)$$
the pro-Galois $G^c(\Zp)$-bundle that gives rise to $\omega_{et}$.
Now recall the diagram of Shimura data, Cartesian as a diagram of group schemes
$$
\begin{CD}
    (\cB,\fX_\cB) @>>> (E^*,h_E)  \\
@VVV        @VVV\\
     (G,\fX) @>>> (G^{ab},h^{ab}).
\end{CD}
$$
This induces maps $\bP_{\cB(\Zp)} \rightarrow a^*\bP_{G(\Zp)}$ and $\bP_{\cB(\Zp)} \rightarrow \delta^*\bP_{E^*(\Zp)}$ that sit above a natural isomorphism $\theta$ identifying the pushforwards of these Galois bundles to $G^{ab}(\Zp)$ via the pullback of $\bP_{G^{ab}(\Zp)}$. Thus we have a $\cB^c(\Zp)$-equivariant map $$\bP_{\cB(\Zp)} \rightarrow a^*\bP_{G(\Zp)} \times_{\theta} \delta^*\bP_{E^*(\Zp)}$$ and we claim it is an isomorphism. But since such a claim may be checked fpqc locally, it is clear: after passing to a suitable pro-Galois cover (e.g. making the base change $\bP_{\cB(\Zp)} \rightarrow \Sh_{\cB(\Zp)}$ itself) it becomes a trivial Galois bundle, and both sides are $\cB^c(\Zp)$-Galois.

The statement (1) is now immediate because functorially in $R$ we have natural isomorphisms respecting Galois actions, and the filtered $F$-isocrystal $\cB^c$-structures on both sides

\begin{eqnarray*}
\bP_{\cB(\Zp)} \otimes_{\Zp} B_{crys}(\hat{R}) & \cong & (a^*\bP_{G(\Zp)} \otimes_{\Zp} B_{crys}(\hat{R}))\times_{\theta}(\delta^*\bP_{E^*(\Zp)} \otimes_{\Zp} B_{crys}(\hat{R}))  \\
& \cong & a^*\cP_{G(\Zp)}(\hat{R}) \otimes_{\hat{R}} B_{crys}(\hat{R}) \times_{\theta} \delta^*\cP_{E^*(\Zp)}(\hat{R}) \otimes_{\hat{R}} B_{crys}(\hat{R})\\
& \cong & \cP_{\cB(\Zp)}(\hat{R}) \otimes_{\hat{R}} B_{crys}(\hat{R}).
\end{eqnarray*}

Finally we need to conclude (1) for $(G_2,\fX_2)$. But this is immediate from the construction of $\cP_{G_2(\Zp)}$ as a quotient by $\cA_p^{dR,\circ}(\cB)$ and the fact that this group acts equivariantly on the isomorphism above. This suffices for $\cP_{G_2(\Zp),\cO_v}$ and we will extend the result to the last remaining case, where $E_{2,v}$ is smaller than $E_v$, in our remarks below.
\end{proof}

\subsubsection{}
\label{Frobenius descent remark}
Recall the gap in the proof of the main theorem, that we defined a Frobenius on $\cP_{G_2(\Zp),\cO_v}$ but did not show that it descends to $\cO_{E_2,v}$. We also (given this gap) were unable to check above that for $(G_2,\fX_2)$ we have $\omega_{et}$ and $\omega_{crys}$ are associated, but only that they were associated after passing to the finite unramified extension $\cO_v/\cO_{E_2,v}$.

Both these gaps are obviously removed by the following lemma of pure relative $p$-adic Hodge theory, which is essentially a relative version of the statement that a Galois representation is crystalline iff it is crystalline after making an unramified extension.

\begin{lem}
Let $X/W$ be a smooth scheme, $W'/W$ an unramified extension. Suppose $\cE' \in \FCrys_{\fX \otimes W'/W'}[1/p]$ and $\bL \in \Lisse_{\Qp}(X_K)$ are such that $\cE'$ and $\bL':= \bL|_{X \otimes_W W'}$ are associated. 

Then there exists a canonical $\cE \in \FCrys_{\fX/W}[1/p]$ such that $\cE$ and $\bL$ are associated and $\cE\otimes_W W' = \cE'$.
\end{lem}
\begin{proof}
We found it easiest to construct this argument using statements from the recent paper of Tan and Tong \cite{tt}, but a more direct proof is likely possible. We let $K' := W'[1/p]$ denote the obvious finite extension of $K=W[1/p]$.

Note that by \cite[3.22]{tt}, we have that $\bL'$ is \emph{crystalline} in the sense of \cite[3.11]{tt} and that $\cE'$ admits a construction as $\cE' = \D_{crys}(\bL')$ via a Fontaine-style functor. But now we may use the same argument as in the classical result: the local description \cite[3.15]{tt} of this functor in terms of Galois invariants makes it clear that there is a $\Gal(K'/K)$-equivariant relation
$$\D_{crys}(\bL) \otimes_K K' \cong \D_{crys}(\bL'),$$
which in turn allows one to deduce that $\bL$ is crystalline (for example, because the conditions of \cite[3.14]{tt} are now obvious). This result also lets us construct $\cE = \D_{crys}(\bL)$, which is a filtered $F$-isocrystal by \cite[3.18]{tt} (because it coverges on the same radius as $\cE'$ is a true isocrystal and not merely a convergent one), and by the other direction of \cite[3.22]{tt} is associated to $\bL$. 
\end{proof}

\subsection{Application to Galois representations}
Recall that one key application of the theory of Shimura varieties in the Langlands programme is that they furnish one with a well-organised supply of Galois representations that one expects to be those attached to a precise collection of automorphic forms. One obvious application of our results is to deduce $p$-adic Hodge theoretic statements about these Galois representations. Specifically, we have the following statement in the case where the underlying Shimura variety is proper.

\begin{thm}
Suppose $(G,\fX)$ is a Shimura datum of abelian type with $G$ unramified at $p$. Fix $G/\Zp$ a reductive integral model, corresponding to a hyperspecial level at $p$, $K^p \subset G(\Afp)$ open compact, and $\rho: G_{\Zp}^c \rightarrow GL(V)$ some weight, with associated lisse $\Zp$-sheaf $\cL_\rho$. 

Assume $\Sh_{K^pG(\Zp)}$ is proper. Then the Galois representation
$$H^i_{et}(\Sh_{K^pG(\Zp)}(G,\fX)_{\overline{E(G,\fX)}}, \cL_\rho[1/p])$$
is crystalline at all places $v|p$ of $E(G,\fX)$.

Moreover, if the Hodge numbers of $V$ (viewed as a Hodge structure via any element of $\fX$) lie in the range $[0,a]$ and $a+i < p-2$, then the lattice $$H^i_{et}(\Sh_{K^pG(\Zp)}(G,\fX)_{\overline{E(G,\fX)}}, \cL_\rho)$$
is Fontaine-Laffaille with corresponding strongly divisible module given by $$H^i_{crys}(\cS_{K^pG(\Zp)}/W, \cE_\rho),$$
where $\cE_\rho$ is the filtered $F$-crystal coming from the crystalline canonical model.
\end{thm}
\begin{proof}
This follows because by our main theorem we have a crystalline canonical model for $(G,\fX)$ which by (\ref{assoc thm}) is integrally associated to the natural etale $G^c(\Zp)$-torsor on $\Sh_{K^pG(\Zp)}(G,\fX)$. In particular, we may find $\cE_\rho$ associated to $\cL_\rho$, and the first part follows from (\ref{comp Qp}).

The second part follows because the condition forces $a<p-2$ and therefore the `integrally associated' statement tells us that $\cE_\rho$ is such that $\bL(\cE_\rho) = \cL_\rho$. With this in hand, we finish by applying (\ref{comp Zp}).
\end{proof}

\subsubsection{}
Finally, we sketch how to extend this result to a non-proper situation, although at the time of writing it seems not all the necessary results are immediately available to carry this out in full generality, and to give a completely general treatment would in any case likely be rather messy.

First, one needs to ensure that crystalline $p$-adic comparison theorems analogous to (\ref{comp Qp}) and (\ref{comp Zp}) are available for the type of cohomology groups one wishes to study. For example Faltings \cite[\S5]{fal1} also gives comparison theorems for cohomology with compact support or partial compact support. Assuming this is the case, one should also note carefully the type of coefficient sheaves required, and that the existence of the \'etale coefficient sheaf can be deduced from the usual setup (for example, to do this one may wish to take advantage of the existence of a toroidal compactification of the Shimura variety at all levels that retains a Hecke action). We shall call whatever is needed on the crystalline side of the picture a ``log crystal'' (hiding the fact that it's actually a log filtered $F$-(iso) crystal and that we need to define what exactly `log' means).

This assessment made, the task will be to construct a log crystal with $G$-structure, so one must follow the argument of this paper checking the necessary steps are still possible. Note that we expect much of this to be not too difficult given the work of Madapusi Pera \cite{mp} (including work in progress on the abelian type case). Loosely speaking, one must carry out the following tasks.

\begin{enumerate}
\item Define a category of log crystals with $G$-structure. It is clear how to do this as a fibre functor taking values in a category of log crystals defined for example as in \cite[1.3.3]{mp}, and thinking about torsors on log schemes it should also be easy to give this a geometric interpretation.
\item Prepare any necessary ambient compactifications of the integral models $\cS_{K_p}(G,\fX)$ with an extension of the Hecke action. For example \cite{mp} produces smooth proper compactifications with normal crossing divisors at the boundary, and an extension of the Hecke action, in the Hodge type case. We also believe Madapusi Pera has or is working on an argument to extend this result to the abelian type case, and that it should be possible to set up diagrams as in \S\ref{abtype} directly relating the Hodge type compactified integral models to the abelian type ones via their connected components.
\item Repeat the argument for the Hodge type case \S\ref{hodge case} replacing the sheaf $\hat{\cV}$ of relative crystalline cohomology of the universal abelian variety with the corresponding log crystal (for example see \cite[3.2]{mp}). One must check that the $s_{\alpha,dR}$ extend to tensors in the category of log crystals (e.g. if in particular they extend to the boundary as in \cite[4.3.7]{mp}), and that one obtains a log crystal with $G$-structure and an equivariant action of $\cA_p^{dR}(G)$.
\item Follow the argument of \S\ref{abtype} to construct a log crystalline canonical model in the abelian type case. We expect it to be formally exactly the same provided the necessary information has been collected in (2) and (3).
\item Check an analogue of the result about being ``integrally associated'' is true in the new situation. Again, we expect this to be formally very similar.
\end{enumerate}

With the log crystal with $G$-structure in hand, one could then conclude a generalisation of the above theorem by using a $p$-adic comparison theorem.


\begin{thebibliography}{10}


\bibitem{bo1}
  Berthelot, P., Ogus, A.
  \emph{F -isocrystals and de Rham cohomology}. I. Invent. Math., 72(2):159?199, 1983.

\bibitem{blasius}
  Blasius, D.
  \emph{A $p$-adic property of Hodge classes on abelian varieties}, Motives (Seattle, WA, 1991), Proc. Sympos. Pure Math., vol. 55, Amer. Math. Soc., Providence, RI, 1994, pp. 293-308.

\bibitem{bre}
  Breuil, C.
  \emph{Integral p-adic Hodge theory}, Algebraic geometry 2000, Azumino (Hotaka), Adv. Stud. Pure Math., vol. 36, Math. Soc. Japan, Tokyo, 2002, pp. 51-80.

\bibitem{brosh}
  Broshi, M.
  \emph{$G$-Torsors over a Dedekind scheme}, J. of Pure and App. Alg. 217 (2013), 11-19.

\bibitem{del2}
  Deligne, P.
  \emph{Vari\'et\'es de Shimura: interpr\'etation modulaire, et techniques de construction de mod\`eles canoniques}, Automorphic forms, representations and L-functions (Corvallis 1977), Proc. Sympos. Pure Math XXXIII, Amer. Math. Soc, pp. 247-289, 1979.

\bibitem{dm}
  Deligne, P., Milne, J.S.
  \emph{Tannakian categories}. In: Hodge cycles, motives and Shimura varieties. Vol. 900. Lecture Notes in Mathematics.

\bibitem{fal1}
  Faltings, G.
  \emph{Crystalline cohomology and p-adic Galois-representations}, Algebraic analysis, geometry, and
number theory (Baltimore), Johns Hopkins Univ. Press (1989), 25-80.
  
\bibitem{far1}
  Fargues, L.
  \emph{Motives and automorphic forms: the (potentially) abelian case}, note.

\bibitem{ham}
  Hamacher, P.
  \emph{On the Newton stratification in the good reduction of Shimura varieties}, preprint 2016.

\bibitem{kim}
  Kim, W.
 \emph{ The classification of p-divisible groups over 2-adic discrete valuation rings}, Math. Res. Lett. 19 (2012), no. 01, 121-141.

\bibitem{kis1}
  Kisin, M.
  \emph{Crystalline representations and F -crystals}, Algebraic geometry and number theory, Progr. Math., vol. 253, Birkhäuser Boston, Boston, MA, 2006, pp. 459-496.

 \bibitem{kis2}
  Kisin, M.
  \emph{Integral models for Shimura varieties of abelian type}, J. AMS 23(4) (2010), 967-1012.


\bibitem{liu1}
  Liu, T.
  \emph{On lattices in semi-stable representations: A proof of a conjecture of Breuil}, 
Compos. Math., 144(1):61-88, 2008.

\bibitem{lz}
  Liu, R., Zhu, X.
  \emph{Rigidity and a Riemann-Hilbert correspondence for p-adic local systems}, Invent. math. (2017) 207-291.


\bibitem{tl}
  Lovering, T.
  \emph{Integral canonical models for automorphic vector bundles of Abelian Type}, preprint, 2016.
\bibitem{mp}
  Madapusi Pera, K.
  \emph{Toroidal compactifications of integral models of Shimura varieties of Hodge type}, preprint 2015.

\bibitem{milne4}
  Milne, J.S.
  \emph{Shimura varieties and motives}, in ?Motives (Seattle, WA, 1991)?, Vol. 55 of Proc. Sympos. Pure Math., Amer. Math. Soc., Providence, RI, pp. 447?523.
  
\bibitem{milne3}
  Milne, J.S.
  \emph{Canonical Models of (Mixed) Shimura Varieties and Automorphic Vector Bundles}, in Automorphic Forms, Shimura Varieties, and L-functions, Perspectives in Math. 10, 1990, pp283 - 414.

\bibitem{moon}
  Moonen, B.
  \emph{Models of Shimura varieties in mixed characteristics}, Galois representations in arithmetic algebraic geometry (Durham, 1996), London Math. Soc. Lecture Note Ser. 254, Cambridge Univ. Press, pp. 267-350, 1998.

\bibitem{og}
  Ogus, A.
  \emph{F-isocrystals and de Rham cohomology II: convergent isocrystals}, Duke Math. J. 51 (1984), 765-850.


\bibitem{sch}
  Scholze, P.
  \emph{$p$-adic Hodge theory for rigid-analytic varieties}, Forum Math. Pi 1, e1 (2013).

\bibitem{tt}
  Fucheng Tan, Jilong Tong
  \emph{Crystalline comparison isomorphism in $p$-adic Hodge theory}, preprint 2015.

\end{thebibliography}
\end{document}